\documentclass[a4paper,11pt]{amsart}

\usepackage[a4paper,hmargin={2.5cm,2.5cm},vmargin={2.5cm,2.5cm}]{geometry}

\usepackage[leqno]{amsmath}
\usepackage{amssymb,amsthm}
\usepackage[mathscr]{euscript}
\usepackage{bbm}
\usepackage{dsfont}
\usepackage[all,arc]{xy}
\usepackage{txfonts}
\DeclareMathAlphabet{\mathpzc}{OT1}{pzc}{m}{it}

\usepackage{chngcntr}
\counterwithin*{equation}{section}

\theoremstyle{plain}
\newtheorem{theorem}{Theorem}[section]
\newtheorem{lemma}[theorem]{Lemma}
\newtheorem{proposition}[theorem]{Proposition}
\newtheorem{corollary}[theorem]{Corollary}

\theoremstyle{definition}

\newtheorem{examples}[theorem]{Examples}
\newtheorem{example}[theorem]{Example}

\theoremstyle{remark}

\newenvironment{eqcond}{\begin{enumerate}}{\end{enumerate}}

\newcommand{\Rw}{\Rightarrow}

\newcommand{\hrw}{\hookrightarrow}

\newcommand{\fra}{\mathfrak{a}}

\newcommand{\frf}{\mathfrak{f}}
\newcommand{\frg}{\mathfrak{g}}
\newcommand{\frj}{\mathfrak{j}}
\newcommand{\frp}{\mathfrak{p}}
\newcommand{\frq}{\mathfrak{q}}
\newcommand{\fru}{\mathfrak{u}}
\newcommand{\frw}{\mathfrak{w}}
\newcommand{\frv}{\mathfrak{v}}
\newcommand{\frx}{\mathfrak{x}}
\newcommand{\fry}{\mathfrak{y}}

\newcommand{\calA}{\mathcal{A}}
\newcommand{\calB}{\mathcal{B}}

\newcommand{\frX}{\mathfrak{X}}
\newcommand{\frY}{\mathfrak{Y}}

\DeclareMathOperator{\ev}{ev}
\DeclareMathOperator{\Id}{Id}
\DeclareMathOperator{\cl}{cl}
\DeclareMathOperator{\dist}{dist}
\DeclareMathOperator{\Lim}{Lim}
\DeclareMathOperator{\match}{\#}
\DeclareMathOperator{\ForgetToV}{S}
\DeclareMathOperator{\ForgetToVAd}{A}
\DeclareMathOperator{\MFunctor}{M}
\DeclareMathOperator{\can}{can}
\DeclareMathOperator{\yoneda}{\mathpzc{y}}
\DeclareMathOperator{\upc}{\uparrow\!}
\DeclareMathOperator{\downc}{\downarrow\!}

\newcommand{\mate}[1]{\,^\ulcorner\! #1^\urcorner}
\newcommand{\umate}[1]{\,_\llcorner\! #1_\lrcorner}

\newcommand{\fspstr}[2]{\llbracket #1,#2\rrbracket}

\newcommand{\catfont}[1]{\mathsf{#1}}

\newcommand{\V}{\catfont{V}}

\newcommand{\two}{\catfont{2}}
\newcommand{\one}{\catfont{1}}
\newcommand{\quantale}{(\V,\otimes,k)}

\newcommand{\Pplus}{\catfont{P}_{\!\!{_+}}}
\newcommand{\Pmax}{\catfont{P}_{\!\!{_\text{max}}}}

\newcommand{\SET}{\catfont{Set}}

\newcommand{\TOP}{\catfont{Top}}
\newcommand{\AP}{\catfont{App}}
\newcommand{\SUP}{\catfont{Sup}}
\newcommand{\ORD}{\catfont{Ord}}
\newcommand{\MET}{\catfont{Met}}
\newcommand{\UMET}{\catfont{UMet}}

\newcommand{\Mat}[1]{#1\text{-}\catfont{Rel}}
\newcommand{\Mod}[1]{#1\text{-}\catfont{Mod}}

\newcommand{\UMat}[1]{#1\text{-}\catfont{URel}}
\newcommand{\Cat}[1]{#1\text{-}\catfont{Cat}}
\newcommand{\CatSep}[1]{#1\text{-}\catfont{Cat}_\mathrm{sep}}
\newcommand{\CatCompl}[1]{#1\text{-}\catfont{Cat}_\mathrm{cpl}}

\renewcommand{\to}{\longrightarrow}
\renewcommand{\mapsto}{\longmapsto}
\newcommand{\relto}{{\longrightarrow\hspace*{-2.8ex}{\mapstochar}\hspace*{2.6ex}}}
\newcommand{\modto}{{\longrightarrow\hspace*{-2.8ex}{\circ}\hspace*{1.2ex}}}
\newcommand{\kto}{\relbar\joinrel\rightharpoonup}
\newcommand{\krelto}{\,{\kto\hspace*{-2.5ex}{\mapstochar}\hspace*{2.6ex}}}
\newcommand{\kmodto}{\,{\kto\hspace*{-2.8ex}{\circ}\hspace*{1.3ex}}}

\newcommand{\kleisli}{\circ}

\newcommand{\homkleislileft}{\multimapinv}

\newcommand{\homcompleft}{\multimapdot}
\newcommand{\homcompright}{\multimapdotinv}

\newcommand{\Txi}{T_{\!\!_\xi}}

\newcommand{\monadfont}[1]{\mathbbm{#1}}

\newcommand{\mT}{\monadfont{T}}
\newcommand{\mI}{\monadfont{1}}
\newcommand{\mU}{\monadfont{U}}

\newcommand{\monad}{(T,e,m)}

\newcommand{\imonad}{(\Id,1,1)}
\newcommand{\umonad}{(U,e,m)}

\newcommand{\theoryfont}[1]{\mathscr{#1}}

\newcommand{\Tth}{\theoryfont{T}}
\newcommand{\Ith}{\theoryfont{I}}
\newcommand{\Uth}{\theoryfont{U}}
\newcommand{\toptheory}{(\mT,\V,\xi)}
\newcommand{\itheory}{(\mI,\V,1_{\V})}

\newcommand{\BC}{(BC)}

\newcommand{\doo}[1]{\overset{\centerdot}{#1}}
\newcommand{\eps}{\varepsilon}
\newcommand{\op}{\mathrm{op}}
\newcommand{\co}{\mathrm{co}}

\usepackage[a4paper,hypertexnames=false]{hyperref} 

\begin{document}

\title{Lawvere completion and separation via closure}

\author{Dirk Hofmann}
\thanks{The first author acknowledges partial financial assistance by Unidade de Investiga\c{c}\~{a}o e Desenvolvimento Matem\'{a}tica e Aplica\c{c}\~{o}es da Universidade de Aveiro/FCT}
\address{Departamento de Matem\'{a}tica\\ Universidade de Aveiro\\3810-193 Aveiro\\ Portugal}
\email{dirk@ua.pt}

\author{Walter Tholen}
\thanks{The second author acknowledges partial financial assistance by the Natural Sciences and Engineering Council of Canada.}
\address{Department of Mathematics and Statistics\\ York University\\ Toronto, Ontario, Canada\\ M3J 1P3}
\email{tholen@mathstat.yorku.ca}


\dedicatory{Dedicated to Bill Lawvere at the occasion of his seventieth birthday}

\begin{abstract}
For a quantale $\V$, first a closure-theoretic approach to completeness and separation in $\V$-categories is presented. This approach is then generalized to $\Tth$-categories, where $\Tth$ is a topological theory that entails a set monad $\mT$ and a compatible $\mT$-algebra structure on $\V$.
\end{abstract}

\maketitle

\section*{Introduction}

Bill Lawvere's 1973 milestone paper ``Metric spaces, generalized logic, and closed categories'' helped us to detect categorical structures in previously unexpected surroundings. His revolutionary idea was not only to regard individual metric spaces as categories (enriched over the monidal-closed category given by the non-negative extended real half-line, with arrows provided by $\geq$ and tensor by $+$), but also to expose the purely categorical nature of the key concept of the theory, Cauchy completeness. The first step to this end was to disregard metric conditions that actually obscure the categorical intuition. In fact, once one has dropped the symmetry requirement it seems much more natural to regard the metric $d$ of a space $X$ as the categorical hom and, given a Cauchy sequence $(a_n)$ in $X$, to associate with it the pair of functions
\begin{align*}
   \varphi(x)&= \lim d(a_n,x) &\text{and}&&   \psi(x) &= \lim d(x,a_n).
\end{align*}
Lawvere's great insight was to expose these functions as pairs of adjoint (bi)modules whose representabilty as
\begin{align*}
   \varphi(x) &= d(a,x) & \text{and} &&   \psi(x) = d(x,a)
\end{align*}
is facilitated precisely by a limit $a$ for $(a_n)$. Hence, a new notion of completeness for categories enriched over any symmetric monoidal-closed category $\V$ was born. Also in the enriched category context it is often referred to as Cauchy completeness. But since Lawvere's brilliant notion entails no sequences at all, just the representability requirement for bimodules, this name seems to be far-fetched and, contrary to popular belief, was in fact not proposed in his paper. Hence, here we use \emph{L-completeness} instead.

In the first part of this paper we give a quick introduction to $\V$-category theory (see \cite{Kel_EnrCat}) in the special case of a commutative unital quantale $\V$, focussing on the themes of L-completion and \emph{L-separation}. We are not aware of an explicit prior occurrence of the latter notion, and both themes are treated with the help of a new closure operator that arises most naturally in the 2-category $\Cat{\V}$, as follows. Call a $\V$-functor $m:M\to X$  \emph{L-dense} if $f\cdot m = g\cdot m$ implies $f\cong g$ for all $\V$-functors $f,g:X \to Y$; the \emph{L-closure} of a subobject $M$ of $X$ is then the largest subobject $\overline{M}$ of $X$ for which $M\to\overline{M}$ is L-dense. For L-separated $\V$-categories, L-dense simply means
epimorphism. The L-separated reflection of a $\V$-category $X$ is its image under the Yoneda functor $\yoneda:X\to\V^{X^\op} = \hat{X}$, and its L-completion is the L-closure of that image in $\hat{X}$.

The main part of the paper is devoted to a substantial generalization of the first part which, however, without the reader's recalling of the more familiar $\V$-category context, may be hard to motivate, especially in view of the considerable additional ``technical'' difficulties. The quantale $\V$ gets augmented by a \emph{topological theory} $\Tth = \toptheory$ which now entails also a Set-monad $\mT$ and a $\mT$-algebra structure $\xi$ on $\V$, with suitable compatibility conditions (see \cite{Hof_TopTh}). While a $\V$-category $X$ comes with a $\V$-relation $a : X \relto X$ (given by a function $a: X \times X \to \V$), $\Tth$-categories come with a $\V$-relation $a: TX \relto X$ making $X$ a lax $\mT$-algebra. For $\mT$ the ultrafilter monad and $\V = \two$, $\Cat{\Tth}$ provides Barr's \cite{Bar_RelAlg} relational description of the category of topological spaces (which, in turn, was based on Manes' \cite{Man_TriplCompAlg} description of compact Hausdorff spaces); for the same monad but with $\V$ the Lawvere half-line, one obains Lowen's approach spaces \cite{Low_Ap}, as shown by Clementino and Hofmann \cite{CH_TopFeat}.

The $\V$-to-$\Tth$ generalization must necessarily entail the provision of a Yoneda functor for a $\Tth$-category $X$. But what is $X^\op$ supposed to be in this highly asymmetric context? Fortunately, this problem was solved in \cite{CH_Compl}: the underlying set of $X^{\op}$ is $TX$, provided with a suitable $\Tth$-structure. This structure needs to be considered {in addition} to the free $\mT$-algebra structure on $TX$, leading to the surprising fact that the $\Tth$-equivalent of the Yoneda functor of the the familiar $\V$-context has now two equally important facets. Once one has fully understood this ``technical'' part of the general theory, it is in fact rather straightforward to extend the $\V$-categorical results on L-completion and L-separation to $\Tth$-categories, again with the help of the L-closure. We could therefore often keep the proofs in the $\Tth$-context quite short, especially when no new ideas beyond the initial ``Yoneda investment'' are needed.

Completeness of $\V$-categories and the induced topology was also investigated by Flagg \cite{Flagg_QuantCont,Flagg_ComplCont} (who called them $\V$-continuity spaces). An alternative approach to the categories of interest in this paper was presented by Burroni \cite{Bur_TCat}. 

\section{Preliminaries}

\subsection{The quantale V} Throughout the paper we consider a commutative and unital quantale $\V=\quantale$. Hence, $\V$ is a complete lattice with a commutative binary operation $\otimes$ and neutral element $k$, such that $u\otimes(-)$ preserves suprema, for all $u\in\V$. Consequentely, $\V$ has an ``internal hom'' $u\multimap(-)$, given by
\[
 z\le u\multimap v \iff z\otimes u\le v,
\]
for all $z,u,v\in\V$. Sometimes we write $v\multimapinv u$ instead of $u\multimap v$. The quantale is \emph{trivial} when $\V=1$; equivalently, when $k=\bot$ is the bottom element of $\V$. Non-trivial examples of quantales are the two-element chain $\two=(\{0,1\},\wedge,1)$, the extended positive half-line $\Pplus=([0,\infty]^\op,+,0)$, and $\Pmax=([0,\infty]^\op,\max,0)$; here $[0,\infty]^\op=([0,\infty],\ge)$, with the natural $\ge$. (We will use $\bigvee$, $\bigwedge$ to denote suprema, infima in $\V$, but use $\sup$, $\inf$, $\max$, etc.\ when we work in $[0,\infty]$ and refer to the natural order $\le$.)

\subsection{$\V$-relations}\label{VRel}
The category $\Mat{\V}$ has sets as objects, and a morphism $r:X\relto Y$ is simply a function $r:X\times Y\to\V$; its composite with $s:Y\relto Z$ is given by
\[s\cdot r(x,z)=\bigvee_{y\in Y}r(x,y)\otimes s(y,z).\]
There is a functor
\[\SET\to\Mat{\V}\]
which maps objects identically and interprets a map $f:X\to Y$ as a $\V$-relation $f_\circ:X\relto Y$:
\[
f_\circ(x,y)=
\begin{cases}
k &\text{if $f(x)=y$,}\\
\bot &\text{otherwise;}
\end{cases}
\]
we normally write $f$ instead of $f_\circ$. The functor is faithful precisely when $k>\bot$. The hom-sets of $\Mat{\V}$ carry the pointwise order of $\V$, so that $\Mat{\V}$ becomes a 2-category. In fact, $\Mat{\V}$ is $\SUP$-enriched (with $\SUP$ the category if complete lattices and suprema-preserving maps), hence it is a quantaloid. Consequentely, for every $r:X\relto Y$, composition by $r$ in $\Mat{\V}$ from either side has a right adjoint, given by \emph{extensions} and \emph{liftings} respectively:
\begin{align*}
(-)\cdot r \dashv (-)\homcompright r && r\cdot(-)\dashv r\homcompleft (-)\\
\underline{\hspace{1ex}t\cdot r\le s\hspace{1ex}} && \underline{\hspace{1ex}r\cdot r\le s\hspace{1ex}}\\
t\le s\homcompright r && t\le r\homcompleft s\\
\xymatrix{X\ar[rd]|-{\object@{|}}^s\ar[d]|-{\object@{|}}_r & \ \\ Y\ar@{}[ur]|-<<{\le} \ar[r]|-{\object@{|}}_t & Z}
&&
\xymatrix{Y & \\ X\ar[u]|-{\object@{|}}^r\ar@{}[ur]|-<<{\le} &
Z\ar[lu]|-{\object@{|}}_s\ar[l]|-{\object@{|}}^t}\\
s\homcompright r(y,z)=\bigwedge_{x\in X}s(x,z)\multimapinv r(x,y) &&
r\homcompleft s(z,x)=\bigwedge_{y\in Y}r(x,y)\multimap s(z,y)
\end{align*}
$\Mat{\V}$ has a contravariant involution
\[
 (\Mat{\V})^\op\to\Mat{\V}
\]
which maps objects identically and assigns to $r:X\relto Y$ its opposite relation $r^\circ:Y\relto X$. When applied to a map $f=f_\circ$, one obtains $f\dashv f^\circ$ in the 2-category $\Mat{\V}$.

\subsection{$\V$-categories}
A \emph{$\V$-category} $X=(X,a)$ is a set $X$ with a $\V$-relation $a:X\relto X$ satisfying $1_X\le a$, $a\cdot a\le a$; equivalentely,
\begin{align*}
 k\le a(x,x), &&  a(x,y)\otimes a(y,z)\le a(x,z)
\end{align*}
for all $x,y,z\in X$. A $\V$-functor $f:(X,a)\to (Y,b)$ must satisfy $f\cdot a\le b\cdot f$; equivalentely,
\[
 a(x,y)\le b(f(x),f(y))
\]
for all $x,y\in X$. The resulting category $\Cat{\V}$ is the category $\ORD$ of (pre)ordered sets if $\V=\two$, Lawvere's category $\MET$ of (pre)metric spaces if $\V=\Pplus$ (see \cite{Law_MetLogClo}), and the category $\UMET$ of (pre)ultrametric spaces if $\V=\Pmax$. For the trivial quantale one has $\Cat{\one}=\SET$. Furthermore, $\V=(\V,\multimap)$ with its internal hom becomes a $\V$-category.

$\Cat{\V}$ is a symmetric monoidal closed category, with tensor product
\begin{align*}
(X,a)\otimes(Y,b)=(X\times Y,a\otimes b), && a\otimes b((x,y),(x',y'))=a(x,x')\otimes b(y,y'),
\end{align*}
and internal hom
\begin{align*}
(X,a)\multimap(Y,b)=(\Cat{\V}(X,Y),[a,b]), && [a,b](f,g)=\bigwedge_{x\in X}b(f(x),g(x)).
\end{align*}
The $\otimes$-neutral object is $E=(E,k)$ (with a singleton set $E$), which generally must be distinguished from the terminal object $1=(1,\top)$ in $\Cat{\V}$. The internal hom describes the pointwise order if $\V=\two$,  and the usual $\sup$-metric if $\Pplus$,$\Pmax$.

\subsection{$\V$-modules}
The category $\Mod{\V}$ has $\V$-categories as objects, and a morphism $\varphi:(X,a)\modto(Y,b)$ is a $\V$-relation $\varphi:X\relto Y$ with $\varphi\cdot a\le\varphi$ and $b\cdot\varphi\le\varphi$. Since always $\varphi=\varphi\cdot 1_X\le\varphi\cdot a $ and $\varphi=1_Y\cdot\varphi\le b\cdot\varphi$, one actually has  $\varphi\cdot a=\varphi$ and $b\cdot\varphi=\varphi$ for a $\V$-module $\varphi:X\modto Y$. In particular, the $\V$-module $a:X\modto X$ assumes the role of the identity morphism on $X$ in $\Mod{\V}$, and we write $a=1_X^*$, in order not to confuse it with $1_X$ in $\Cat{\V}$. This notation is extended to arbitrary maps $f:X\to Y$ by
\begin{align*}
f_*=b\cdot f, && f^*=f^\circ\cdot b,
\end{align*}
and one easily verifies:
\begin{lemma}
The following are equivalent for a map $f:X\to Y$ between $\V$-categories $X$ and $Y$.
\begin{eqcond}
\item $f:X\to Y$ is a $\V$-functor.
\item $f_*$ is a $\V$-module $f_*:X\modto Y$.
\item $f^*$ is a $\V$-module $f^*:Y\modto X$.
\end{eqcond}
\end{lemma}
Hence there are functors which make the following diagram commute.
\[
\xymatrix{\Cat{\V}\ar[r]^{(-)_*} & \Mod{\V} & (\Cat{\V})^\op\ar[l]_{(-)^*}\\
\SET\ar[u]\ar[r]_{(-)_\circ} & \Mat{\V}\ar[u] & \SET^\op\ar[l]^{(-)^\circ}\ar[u]}
\]
Here the vertical full embeddings are given by $X\mapsto (X,1_X)$. Just like $\Mat{\V}$ also $\Mod{\V}$ is a quantaloid, with the same pointwise order structure. But not just suprema of $\V$-modules formed in $\Mat{\V}$ are again $\V$-modules, also extensions and liftings. For example, for $\varphi:(X,a)\modto(Y,b)$, $\psi:(Z,c)\modto(Y,b)$, the lifting $\varphi\homcompleft\psi$ formed in $\Mat{\V}$ is indeed a $\V$-module $\varphi\homcompleft\psi:(Z,c)\to(Y,b)$: from $\psi\cdot c\le\psi$ and $\varphi\cdot(\varphi\homcompleft\psi)\le\psi$ one obtains $\varphi\cdot(\varphi\homcompleft\psi)\cdot c\le \psi$ and then $(\varphi\homcompleft\psi)\cdot c\le\varphi\homcompleft\psi$; similar $a\cdot(\varphi\homcompleft\psi)\le\varphi\homcompleft\psi$. Also the contravariant involution of $\Mat{\V}$ extends to $\Mod{\V}$ (e.g., if $\varphi:X\modto Y$, then $\varphi:X^\op\modto Y^\op$, where $X^\op=(X,a^\circ)$ is the usual opposite $\V$-category.), and one has the commutative diagram.
\[
\xymatrix{(\Mod{\V})^\op\ar[r]^-{(-)^\op} & \Mod{\V}\\ \Cat{\V}\ar[r]_{(-)^\op}\ar[u]^{(-)^*} & \Cat{\V}\ar[u]_{(-)_*}}
\]
As a quantaloid, $\Mod{\V}$ is in particular a 2-category, and for all $f:X\to Y$ in $\Cat{\V}$ one has
\[
 f_*\dashv f^*
\]
in $\Mod{\V}$. $\Cat{\V}$ inherits its 2-categorical structure from $\Mod{\V}$ via
\begin{align*}
f\le f' &:\iff f^*\le(f')^*\iff \forall x\in X,y\in Y\,.\,b(y,f(x))\le b(y,f'(x))\\
 &\iff f'_*\le f_*\iff \forall x\in X, y\in Y\,.\, b(f'(x),y)\le b(f(x),y)\\
 &\iff 1_X^*\le (f')^*\cdot f_*\iff \forall x\in X\,.\,k\le b(f(x),f'(x)).
\end{align*}
Hence, the previous diagram actually shows commuting 2-functors when we add dualization w.r.t.\ 2-cells (indicated by co) appropriately:
\[
\xymatrix{(\Mod{\V})^\op\ar[r]^-{(-)^\op} & \Mod{\V}\\ \Cat{\V}\ar[r]_{(-)^\op}\ar[u]^{(-)^*} & \Cat{\V}^\co\ar[u]_{(-)_*}}
\]
Of course, $\Cat{\V}$ being a 2-category, there is also a notion of adjointness in $\Cat{\V}$:
\begin{align*}
f\dashv g \text{ in $\Cat{\V}$} &\iff f\cdot g\le 1\text{ and } 1\le g\cdot f\text{ in $\Cat{\V}$}\\
&\iff g^*\cdot f^*\le 1^*\text{ and }1^*\le f^*\cdot g^*\text{ in $\Mod{\V}$}\\
&\iff g^*\dashv f^* \text{ in $\Mod{\V}$}\\
&\iff f_*=g^* \hspace{1em}\text{(since $f_*\dashv f^*$ in $\Mod{\V}$)}\\
&\iff g_*=f_* \hspace{1em}\text{(since $g_*\dashv g^*$ in $\Mod{\V}$)}\\
&\iff \forall x\in X, y\in Y\,.\,a(x,g(y))=b(f(x),y).
\end{align*}

\subsection{Yoneda}
$\V$-modules give rise to $\V$-functors, as follows
\begin{proposition}\label{CharModV}
The following are equivalent for $\V$-relations $\varphi:X\relto Y$ between $\V$-categories:
\begin{eqcond}
\item $\varphi:X\modto Y$ is a $\V$-module.
\item $\varphi:X^\op\otimes Y\to\V$ is a $\V$-functor.
\end{eqcond}
\end{proposition}
With $\varphi=a=1_X^*:X\modto X$ we obtain in particular the $\V$-functor $a:X^\op\otimes X\to\V$ whose transpose $\mate{a}$ is the Yoneda-$\V$-functor
\[
\yoneda:X\to\hat{X}:=(X^\op\multimap\V),\;x\mapsto a(-,x).
\]
The structure $\hat{a}$ of $\hat{X}$ is given by
\[
 \hat{a}(f,f')=\bigwedge_{x\in X}f(x)\multimap f'(x).
\]
\begin{lemma}\label{YonedaLemV}
For all $x\in X$ and $f\in\hat{X}$, $\hat{a}(\yoneda(x),f)=f(x)$.
\end{lemma}
One calls a $\V$-functor $f:(X,a)\to(Y,b)$ \emph{fully faithful} if $a(x,x')=b(f(x),f(x'))$ for all $x,x'\in X$; equivalently, if $1_X^*=f^*\cdot f_*$ (since $f^*\cdot f_*=f^\circ\cdot b\cdot f$), or just $1_X^*\ge f^*\cdot f_*$ (since the other inequality comes for free).
\begin{corollary}
$\yoneda:X\to\hat{X}$ is fully faithful.
\end{corollary}

\subsection{L-separation}\label{LsepV}
For $\V$-functors $f,g:Z\to X$ we write $f\cong g$ if $f\le g$ and $g\le f$; equivalently, if $f^*=g^*$, or $f_*=g_*$. We call $X$ \emph{L-separated} if $f\cong g$ implies $f=g$, for all $f,g:Z\to X$. The full subcategory of $\Cat{\V}$ consisting of all L-separated $\V$-categories is denoted by $\CatSep{\V}$. Obviously, it suffices to consider $Z=E$ (the $\otimes$-neutral object) here: writing $x:E\to X$ in $\Cat{\V}$ instead of $x\in X$, we just note that $f_*=g_*$ implies
\[
(f\cdot x)_*=f_*\cdot x_*=g_*\cdot x_*=(g\cdot x)_*.
\]
This proves the equivalence of (i),(ii) of the following proposition.
\begin{proposition}\label{EqCondSepV}
The following statements are equivalent for a $\V$-category $X=(X,a)$.
\begin{eqcond}
\item $X$ is L-separated.
\item $x\cong y$ implies $x=y$, for all $x,y\in X$.
\item For all $x,y\in X$, if $a(x,y)\ge k$ and $a(y,x)\ge k$, then $x=y$.
\item The Yoneda functor $\yoneda:X\to\hat{X}$ is injective.
\end{eqcond}
\end{proposition}
\begin{proof}
For (ii)$\iff$(iii)$\iff$(iv) one observes
\begin{align*}
\yoneda(x)=\yoneda(y) &\iff x^\circ\cdot a=y^\circ\cdot a\\
&\iff x^*=y^*\\
&\iff x\le y\text{ and } y\le x\\
&\iff k\le a(x,y)\text{ and } k\le a(y,x).\qedhere
\end{align*}
\end{proof}
\begin{corollary}
The $\V$-category $\V$ is L-separated. For all $\V$-categories $X$, $Y$, if $Y$ is L-separated, $X\multimap Y$ is also L-separated. In particular, $\hat{X}$ is L-separated, for every $X$.
\end{corollary}
\begin{proof}
$k\le u\multimap v$ and $k\le v\multimap u$ means $u\le v$ and $v\le u$ in $\V$, hence $u=v$. For $Y=(Y,b)$ and $X=(X,a)$, $k\le [a,b](f,g)$ in $X\multimap Y$ means $k\le b(f(x),g(x))$ for all $x\in X$, which makes the second statement obvious.
\end{proof}

\subsection{L-completeness}\label{ComplV}
Following Lawvere \cite{Law_MetLogClo} we call a $\V$-category $X$ \emph{L-complete} if every adjunction $\varphi\dashv\psi:X\modto Z$ in $\Mod{\V}$ is of the form $f_*\dashv f^*$, for a $\V$-functor $f:Z\to X$. Clearly, if $X$ is L-separated, such a presentation is unique. As in \ref{LsepV}, it suffices to consider $Z=E$ here; but we need the Axiom of Choice here.
\begin{proposition}
The following statements are equivalent for a $\V$-category $X$.
\begin{eqcond}
\item $X$ is L-complete.
\item Each left adjoint $\V$-module $\varphi:E\modto X$ is of the form $\varphi=x_*$ for some $x$ in $X$.
\item Each right adjoint $\V$-module $\psi:X\modto E$ is of the form $\psi=x^*$ for some $x$ in $X$.
\end{eqcond}
\end{proposition}
Elements in $\hat{X}$ are $\V$-functors $X^\op\cong X^\op\otimes E\to\V$ which, by Proposition \ref{CharModV}, may be considered as a $\V$-module $\psi:X\modto E$. Suppose this $\V$-module has a left adjoint $\varphi:E\modto X$. From $\varphi\cdot\psi\le 1_X^*$ one obtains $\varphi\le 1_X^*\homcompright\psi$ (see \ref{VRel}), and from $(1_X^*\homcompright\psi)\cdot\psi\le 1_X^*$ and $\psi\cdot\varphi\ge 1_E^*$ one has $1_X^*\homcompright\psi\le\varphi$. Hence, if $\psi$ is right adjoint, its left adjoint must necessarily be $1_X^*\homcompright\psi$; moreover $(1_X^*\homcompright\psi)\cdot\psi\le 1_X^*$ always holds. Therefore:
\begin{proposition}
A $\V$-module $\psi:X\modto E$ (with $X=(X,a)$) is right adjoint if, and only if, $1_E^*\le\psi\cdot  (1_X^*\homcompright\psi)$, that is, if
\begin{equation}\label{tightV}
 k\le\bigvee_{y\in Y}\psi(y)\otimes\left(\bigwedge_{x\in X}a(x,y)\multimapinv\psi(x)\right).
\end{equation}
\end{proposition}
Note that $\bigwedge_{x\in X}a(x,y)\multimapinv\psi(x)=\hat{a}(\psi,\yoneda(y))$. We call a $\V$-functor $\psi:X^\op\to\V$ \emph{tight} if, as a $\V$-module $X\modto E$, it is right adjoint, that is, if it satisfies \eqref{tightV}. We consider
\[
 \tilde{X}=\{\psi\in\hat{X}\mid \psi\text{ tight}\}
\]
as a full $\V$-subcategory of $\hat{X}$. Our goal is to exhibit $\tilde{X}$ as an ``L-completion'' of $X$. 
\begin{examples}
(1) $\V=\two$. A $\V$-functor $X^\op\to\two$ is the characteristic function of a down-closed set $A$ in the (pre)ordered set $X$. Condition \eqref{tightV} the reads as
\[
 \exists y\in A\forall x\in A\,.\,x\le y,
\]
so that $A=\downc y$. In other words, $\tilde{X}$ is simply the image of the Yoneda functor $\yoneda:X\to\hat{X},\, y\mapsto\downc y$.\\
(2) $\V=\Pplus$. A tight $\V$-functor $X^\op\to\V$ is given by a function $\psi:X\to[0,\infty]$ with
\begin{align*}
& \psi(y)\le\psi(x)\;\Rw\; \psi(x)-\psi(y)\le a(x,y)\hspace{1em}(x,y\in X),\\
& \inf_{y\in Y}(\psi(y)+\sup_{\substack{x\in X,\\\psi(x)\le a(x,y)}}(a(x,y)-\psi(x)))=0, 
\end{align*}
here $a$ is the generalized metric on $X$. If $a$ is symmetric (so that $a=a^\circ$), these conditions are more conveniently describes as
\begin{align*}
& |\psi(x)-\psi(y)|\le a(x,y)\le\psi(x)+\psi(y)\hspace{1em}(x,y\in X),\\
& \inf_{x\in X}\psi(x)=0.
\end{align*}
These are precisely the \emph{supertight} maps on $X$ considered in \cite{LS_Compl}.\\
(3) $\V=\Pmax$. Here the two conditions of (2) change to
\begin{align*}
& \psi(y)<\psi(x)\;\Rw\;\psi(x)\le a(x,y)\hspace{1em}(x,y\in X),\\
& \inf_{y\in Y}(\max(\psi(y),\sup_{\substack{x\in X,\\\psi(x)< a(x,y)}}(a(x,y))))=0.
\end{align*}
\end{examples}

\subsection{L-injectivity}\label{LInjV}
A $\V$-functor $f:(X,a)\to (Y,b)$ is called \emph{L-dense} if $f_*\cdot f^*=1_Y^*$; that is, if $b=b\cdot f\cdot f^\circ\cdot b$, or
\[
b(y,y')=\bigvee_{x\in X} b(y,f(x))\otimes b(f(x),y')
\]
for all $y,y'\in Y$. L-dense $\V$-functors have good composition-cancellation properties.
\begin{lemma}\label{CompCancelV}
Let $f:X\to Y$ and $g:Y\to Z$ be $\V$-functors. Then the following assertions hold.
\begin{enumerate}
\item $f$, $g$ L-dense $\Rw$ $g\cdot f$ L-dense.
\item $g\cdot f$ L-dense $\Rw$ $g$ L-dense.
\item $g\cdot f$ L-dense, $g$ fully faithful $\Rw$ $f$ L-dense.
\item $g\cdot f$ fully faithful, $f$ L-dense $\Rw$ $g$ fully faithful.
\end{enumerate}
\end{lemma}
A fully faithful L-dense $\V$-functor is an \emph{L-equivalence}. Hence, $f$ is an L-equivalence if, and only if, $f_*$ (or $f^*$) is an isomorphism in $\Mod{\V}$. A $\V$-category $Z$ is \emph{pseudo-injective} if, for every fully faithful $\V$-functor $f:X\to Y$ and for all $\V$-functors $h:X\to Z$ there is a $\V$-functor $g:Y\to Z$ with $g\cdot f\cong h$; if strict equality is obtainable, we call $Z$ injective. $Z$ is \emph{L-injective} if this extension property is required only for L-equivalences $f$. Hence, injectivity implies pseudo-injectivity, and every pseudo-injective $V$-category is also L-injective.
\begin{lemma}
The $\V$-category $\V$ is injective, hence in particular L-injective.
\end{lemma}
\begin{proof}
Let $f:X\to Y$ be fully faithful and $\varphi:X\to\V$ be any $\V$-functor. Then the $\V$-module $\varphi:E\modto X$ factors as $\varphi=f^*\cdot\psi$, with $\psi=f_*\cdot\varphi$. But the $\V$-module $f^*\cdot^\psi$ corresponds to the $\V$-functor $\psi\cdot f$, hence $\psi\cdot f=\varphi$.
\begin{align*}
\xymatrix{& \V\\ X\ar[r]_f\ar[ur]^\varphi & Y\ar[u]_\psi} &
\xymatrix{Y\ar@{-^{>}}|-{\object@{o}}[r]^{f^*} & X\\ E\ar@{-^{>}}|-{\object@{o}}[u]^{\psi=f_*\cdot\varphi}\ar@{-^{>}}|-{\object@{o}}[ur]_\varphi}
\end{align*}\
\end{proof}
Note that the $\V$-functor $\psi$ has been constructed effectively, with 
\[\psi(y)=\bigvee_{x\in X}\varphi(x)\otimes b(f(x),y).\]
In case $\V=\two$, this means
\[
\psi(y)=\top\iff \exists x\in X\,.\,(\varphi(x)=\top\text{ and }f(x)=y),
\]
and for $\V=\Pplus$ we have
\[\psi(y)=\inf_{x\in X}(\varphi(x)+b(f(x),y)).\]
\begin{proposition}
For all $\V$-categories $X$, $Y$, if $Y$ is pseudo-injective or L-injective, $X\multimap Y$ has the respective property. In particular, $\hat{X}$ is injective.
\end{proposition}
\begin{proof}
Let $f:A\to B$ be a fully faithful, and consider any $\V$-functor $\varphi:A\to(X\multimap Y)$, with $Y$ pseudo-injective. Since $f\otimes 1_X$ is fully faithful, the mate $\umate{\varphi}:A\otimes X\to Y$ factors (up to $\cong$) as $\umate{\varphi}\cong\umate{\psi}\cdot(f\otimes 1_X)$, with $\umate{\psi}:B\otimes X\to Y$ corresponding to a $\V$-functor $\psi:B\to(X\multimap Y)$. Since $\umate{\psi}\cdot(f\otimes 1_X)$ corresponds to $\psi\cdot f$, $\varphi\cong\psi\cdot f$ follows. The proof works \emph{mutatis mutandis} for L-injectivity.
\end{proof}
Our goal is to show that L-injectivity and L-completeness are equivalent properties.

\section{L-closure}

\subsection{L-dense $\V$-functors}\label{LdenseV}
We first show that L-dense $\V$-functors are characterized as ``epimorphisms up to $\cong$''. 
\begin{proposition}
A $\V$-functor $m:M\to X$ is L-dense if, and only if, for all $\V$-functors $f,g:X\to Y$ with $f\cdot m=g\cdot m$ one has $f\cong g$.
\end{proposition}
\begin{proof}
The necessity of the condition is clear since from $f_*\cdot m_*=g_*\cdot m_*$ one obtains $f_*=g_*$ when $m_*\cdot m^*=1_X^*$. To show the converse implication, by Lemma \ref{CompCancelV} we may assume that $m$ is a full embedding $M\hrw X$ and consider its cokernel pair 
\[
\xymatrix{(X,a)\ar@<0.5ex>[r]^f\ar@<-0.5ex>[r]_g & (Y,b),}
\]
given by the disjoint union
\[
Y=\{f(x)=g(x)\mid x\in M\}\cup \{f(x)\mid x\in X\setminus M\}\cup\{g(x)\mid x\in X\setminus M\},
\]
where both $f$ and $g$ are full embeddings, and
\[
b(f(x),g(y))=\bigvee_{z\in Z}a(x,z)\otimes a(z,y)
\]
for all $y,x\in X\setminus M$. Since $f_*=g_*$ by hypothesis, we obtain
\[
a(x,y)=b(g(x),g(y))=b(f(x),g(y))=m_*\cdot m^*(x,y)
\]
for all $x,y\in X\setminus M$. But this identity holds trivially when $x\in M$ or $y\in M$. Hence $m_*\cdot m^*=1_X^*$.
\end{proof}
Since $f\cong g$ precisly when $f\cdot x\cong g\cdot x$ for all $x\in X$ (considered as $x:E\to X$), it is now easy to identify the largest subset of $X$ which contains $M$ as an L-dense subset.

\subsection{L-closure}\label{LclosureV}
For a $\V$-category $X$ and $M\subseteq X$, we define the \emph{L-closure} of $M$ in $X$ by
\[
\overline{M}=\{x\in X\mid \forall f,g:X\to Y\,.\,(f|_M=g|_M\,\Rw\,f\cdot x\cong g\cdot x)\}
\]
and prove
\begin{proposition}\label{CharLClsV}
Let $X=(X,a)$ be a $\V$-category, $M\subseteq X$ and $x\in X$. Then the following assertions are equivalent.
\begin{eqcond}
\item $x\in\overline{M}$.
\item $a(x,x)\le\bigvee_{y\in M}a(x,y)\otimes a(y,x)$
\item $k\le\bigvee_{y\in M}a(x,y)\otimes a(y,x)$
\item $1_E^*\le x^*\cdot m_*\cdot m^*\cdot x_*$, 
\item $m^*\cdot x_*\dashv x^*\cdot m_*$, where $m$ denotes the full embedding $m:M\hrw X$.
\item $x_*:E\modto X$ factors through $m_*:M\modto X$ by a map $\varphi:E\modto M$ in $\Mod{\V}$. 
\end{eqcond}
\end{proposition}
\begin{proof}
(i)$\Rw$(ii) follows from $M\hrw\overline{M}$ is dense. (ii)$\Rw$(iii) is clear since $k\le a(x,x)$. To see (iii)$\Rw$(iv), just observe that
\[
x^*\cdot m_*\cdot m^*\cdot x_*(\star,\star)=\bigvee_{y\in M}a(x,y)\otimes a(y,x).
\]
Since $m^*\cdot x_*\cdot x^*\cdot m_*\le m^*\cdot m_*=1^*_M$, (iv)$\Rw$(v). Assuming (v), we have $m_*\cdot m^*\cdot x_*\dashv x^*\cdot m_*\cdot x^*$ as well as $m_*\cdot m^*\cdot x_*\le x_*$ and $x^*\cdot m_*\cdot x^*\le x^*$, which implies $m_*\cdot m^*\cdot x_*= x_*$ and we have shown (v)$\Rw$(vi). Finally, assume (vi) and let $f,g:X\to Y$ with $f|_M=g|_M$. Then
\[
f_*\cdot x_*=f_*\cdot m_*\cdot\varphi=g_*\cdot m_*\cdot\varphi=g_*\cdot x_*,
\]
which proves (i).
\end{proof}
$\V$-functors respect the L-closure, as we show next.
\begin{proposition}
For a $\Tth$-functor $f : X\to Y$ and $M,M'\subseteq X$, $N\subseteq Y$, we have
\begin{enumerate}
\item $M\subseteq\overline{M}$ and $M\subseteq M'$ implies $\overline{M}\subseteq\overline{M'}$.
\item $\overline{\varnothing}=\varnothing$ and $\overline{\overline{M}}=\overline{M}$.
\item $f(\overline{M})\subseteq \overline{f(M)}$ and $f^{-1}(\overline{N})\supseteq\overline{f^{-1}(N)}$.
\item If $k$ is $\vee$-irreducible (so that $k\le u\vee v$ implies $k\le u$ or $k\le v$), then $\overline{M\cup M'}=\overline{M}\cup\overline{M'}$. 
\end{enumerate}
\end{proposition}
\begin{proof}
(1), (2) are obvious. For (3), applying Lemma \ref{CompCancelV} to
\[
\xymatrix{M\ar[r]\ar@{>->}[d] & f(M)\ar@{>->}[d] \\ \overline{M}\ar[r] & f(\overline{M})}
\]
one sees that $f(M)\to f(\overline{M})$ is L-dense, hence $f(\overline{M})\subseteq\overline{f(M)}$. With $M=f^{-1}(N)$, this implies $\overline{f^{-1}(N)}\subseteq f^{-1}(\overline{N})$.
To see (4), we just need to show that $x\in\overline{M\cup M'}$ implies $x\in\overline{M}$ or $x\in\overline{M'}$. But this follows from
\[
k\le \bigvee_{y\in M\cup M'}a(x,y)\otimes a(y,x)=
\left( \bigvee_{y\in M}a(x,y)\otimes a(y,x)\right) \vee
\left( \bigvee_{y\in M'}a(x,y)\otimes a(y,x)\right),
\]
assuming that $k$ is $\vee$-irreducible.
\end{proof}
\begin{corollary}
If $k$ is $\vee$-irreducible in $\V$, then the L-closure operator defines a topology on $X$ such that every $\V$-functor becomes continuous. Hence, L-closure defines a functor $L:\Cat{\V}\to\TOP$.
\end{corollary}
\begin{examples}
(1)\hspace{1em} For $X=(X,\le)$ in $\Cat{\two}=\ORD$ and $M\subseteq X$, one has $x\in \overline{M}$ precisely when $x\le z\le x$ for some $z\in M$. Also, $M\subseteq X$ is open in $LX$ if every $x\in M$ satisfies
\[
\forall z\in X\,.\,(x\le z\le x\;\Rw\;z\in M).
\]
(2)\hspace{1em} In $\MET$, $\overline{M}=\{x\in X=(X,a)\mid \inf_{z\in M}(a(x,z)+a(z,x))=0\}$, and in $\UMET$ \[\overline{M}=\{x\in X=(X,a)\mid \inf_{z\in M}(\max(a(x,z),a(z,x)))=0\}\]
which for symmetric (ultra)metric spaces describes the ordinary topological closure.  
\end{examples}

\subsection{L-separatedness via the L-closure}
\begin{proposition}\label{ClDiagV}
Let $X=(X,a)$ be a $\V$-category and $\Delta\subseteq X\times X$ the diagonal. Then
\[
\overline{\Delta}=\{(x,y)\in X\times Y\mid x\cong y\}
\]
\end{proposition}
\begin{proof}
Let first $(x,y)\in\overline{\Delta}$. With $\pi_1,\pi_2:X\times X\to X$ denoting the projection maps, we have $\pi_1|_\Delta=\pi_2|_\Delta$ and therefore $x=\pi_1(x,y)\cong\pi_2(x,y)=y$. Assume now $x\cong y$. Note that the canonical functor $\Cat{\V}\to\ORD$ preserves products, hence 
\begin{align*}
(x_1,y_1)\cong(x_2,y_2) &\iff x_1\cong x_2\text{ and }y_1\cong y_2, 
\end{align*}
for all $(x_1,y_1),(x_2,y_2)\in X\times X$. Therefore we have $(x,y)\cong (x,x)$. Let now $f,g:X\times X\to Y$ be $\V$-functors with $f|_\Delta=g|_\Delta$. Then
$f(x,y)\cong f(x,x)=g(x,x)\cong g(x,y)$.
\end{proof}
\begin{corollary}
A $\V$-category $X$ is L-separated if and only if the diagonal $\Delta$ is closed in $X\times X$. 
\end{corollary}
\begin{theorem}
$\CatSep{\V}$ is an epi-reflective subcategory of $\Cat{\V}$, where the reflection map is given by $\yoneda_X:X\to\yoneda_X(X)$, for each $\V$-category $X$. Hence, limits of L-separated $\V$-categories are formed in $\Cat{\V}$, while colimits are obtained by reflecting the colimit formed in $\Cat{\V}$. The epimorphisms in $\Cat{\V}$ are precisely the L-dense $\V$-functors. 
\end{theorem}

\subsection{L-completeness via the L-closure}
\begin{lemma}\label{ClosVsComplV}
Let $X=(X,a)$ be a $\V$-category and $M\subseteq X$. Then the following assertions hold.
\begin{enumerate}
\item Assume that $X$ is L-complete and $M$ be L-closed. Then $M$ is L-complete.
\item Assume that $X$ is L-separated and $M$ is L-complete. Then $M$ is L-closed.
\end{enumerate}
\end{lemma}
\begin{proof}
(1) follows immediately from Proposition \ref{CharLClsV}. To see (2), let $x\in X$ such that $m^*\cdot x_*\dashv x^*\cdot m_*$. Since $M$ is L-complete, there is some $y\in M$ such that $y_*=m^*\cdot x_*$ and $y^*=x^*\cdot m_*$. Hence $m(y)_*=m_*\cdot y_*\le x_*$ and $m(y)^*=i^*\cdot y^*\le x^*$ and therefore, $m(y)_*=x_*$. L-separatedness of $X$ gives now $m(y)=x$, i.e.\ $x\in M$.
\end{proof}
\begin{theorem}
Let $X=(X,b)$ be a $\V$-category. The following assertions are equivalent.
\begin{eqcond}
\item $X$ is L-complete.
\item $X$ is L-injective.
\item $\yoneda:X\to\tilde{X}$ has a pseudo left-inverse $\V$-functor $R:\tilde{X}\to X$, i.e.\ $R\cdot\yoneda\cong 1_X$.
\end{eqcond}
\end{theorem}
\begin{proof}
As for Theorem \ref{CharComplT}.
\end{proof}
\begin{proposition}
For a $\V$-category $X$, as a set $\tilde{X}$ (see \ref{ComplV}) coincides with the L-closure of $\yoneda(X)$ in $\hat{X}$. Hence, $\yoneda:X\to\tilde{X}$ is fully faithful and L-dense, and $\tilde{X}$ is L-complete.
\end{proposition}
\begin{proof}
By Proposition \ref{CharLClsV}, a $\V$-functor $\psi:X^\op\to\V$ lies in the L-closure of $\yoneda(X)$ in $\hat{X}$ if, and only if,
\[
k\le\bigvee_{y\in X}\hat{a}(\psi,\yoneda(y))\otimes\hat{a}(\yoneda(y),\psi).
\]
Since $\hat{a}(\yoneda(y),\psi)=\psi(y)$ by Lemma \ref{YonedaLemV}, this means precisely that $\psi$ must be tight.
\end{proof}
\begin{theorem}
The full subcategory $\CatCompl{\V}$ of $\CatSep{\V}$ of L-complete $\V$-categories is an epi-reflective subcategory of $\CatSep{\V}$. The reflection map of a L-separated $\V$-category $X$ is given by any dense embedding of $X$ into a L-complete and L-separated $\V$-category, for instance by $\yoneda:X\to\tilde{X}$.
\end{theorem}

\section{The $\Tth$-setting}

\subsection{The theory $\Tth$}\label{Topth}
From now on we assume that the quantale $\V$ is part of a \emph{strict topological theory} $\Tth=\toptheory$ as introduced in \cite{Hof_TopTh}. Here $\mT=\monad$ is a $\SET$-monad where $T$ and $m$ satisfy \BC\ (that is, $T$ sends pullbacks to weak pullbacks and each naturality square of $m$ is a weak pullback) and $\xi:T\V\to\V$ is a map such that
\begin{align*}
 &\hspace{2em} 1_{\V}=\xi\cdot e_{\V}, &&
 &\xi\cdot T\xi=\xi\cdot m_{\V},\\
\intertext{the diagrams}
 &\hspace{-2em}
\xymatrix{T(\V\times\V)\ar[rr]^-{T(\otimes)}\ar[d]_{\langle\xi\cdot T\pi_1,\xi\cdot T\pi_2\rangle} && T\V\ar[d]^\xi\\ \V\times\V\ar[rr]_-{\otimes} && \V,} &&
 &
\xymatrix{T1\ar[d]_{!}\ar[r]^{Tk} & T\V\ar[d]^\xi\\ 1\ar[r]_k & \V,}
\end{align*}
commute and 
\begin{itemize}
\item[] $(\xi_{_X})_{X}:P_{\V}\to P_{\V} T$ is a natural transformation, where $P_\V$ is the $\V$-powerset functor considered as a functor from $\SET$ to $\ORD$ and the $X$-component $\xi_X:P_\V(X)\to P_\V T(X)$ is given by $\varphi\mapsto \xi\cdot T\varphi$.
\end{itemize}
Here $P_\V(X)=\V^X$, and for a function $f:X\to Y$ we have a canonical map $f^{-1}:\V^Y\to\V^X,\,\varphi\mapsto\varphi\cdot f$.  Now $P_\V(f)$ is defined as the left adjoint to $f^{-1}$, explicitly, for $\varphi\in\V^X$ we have $P_\V(\varphi)(y)=\bigvee_{x\in f^{-1}(y)}\varphi(x)$.
Furthermore, we assume $T1=1$. 
 \begin{examples}\label{ExTheories}
\begin{enumerate}
\item\label{ExId} For each quantale $\V$, $\itheory$ is a strict topological theory, where $\mI=\imonad$ denotes the identity monad.
\item\label{ExTop} $\Uth_{\two}=(\mU,\two,\xi_{\two})$ is a strict topological theory, where $\mU=\umonad$ denotes the ultrafilter monad and $\xi_{\two}$ is essentially the identity map.
\item\label{ExApp} $\Uth_{\Pplus}=(\mU,\Pplus,\xi_{\Pplus})$ is a strict topological theory, where
\[
\xi_{\Pplus}:U\Pplus\to\Pplus,\;\;\frx\mapsto\inf\{v\in\Pplus\mid[0,v]\in\frx\}.
\]
\end{enumerate}
\end{examples}
As shown in \cite[Lemma 3.2]{Hof_TopTh}, the right adjoint $\multimap$ of the tensor product $\otimes$ in $\V$ is automatically compatible with the map $\xi:T\V\to\V$ in the sense that
\begin{gather*}
\xi\cdot T(\multimap)\le\multimap\cdot\langle\xi\cdot T\pi_1,\xi\cdot T\pi_2\rangle.\\
\xymatrix{T(\V\times\V)\ar[rr]^-{T(\multimap)}\ar[d]_{\langle\xi\cdot T\pi_1,\xi\cdot T\pi_2\rangle}\ar@{}[drr]|{\ge} && T\V\ar[d]^\xi\\ \V\times\V\ar[rr]_-{\multimap} && \V}
\end{gather*}
Furthermore, our condition $T1=1$ implies $m_X^\circ\cdot e_X=e_{TX}\cdot e_X$ for each set $X$. In fact, $m_X^\circ\cdot e_X\ge e_{TX}\cdot e_X$ is true for each monad since $m_X^\circ\ge e_{TX}$. Let now $\frX\in TTX$ and $x\in X$ such that $m_X(\frX)=e_X(x)$. We consider the commutative diagram
\[
\xymatrix{TT1\ar[d]_{m_1}\ar[r]^{TTx} & TTX\ar[d]^{m_X}\\ T1\ar[r]_{Tx} & TX,}
\]
where $x:1\mapsto X$. Since $m$ satisfies \BC, there is some $\frY\in TT1=1$ with $TTx(\frY)=\frX$, that is, $\frX=e_{TX}\cdot e_X(x)$.

The functor $T:\SET\to\SET$ can be extended to a 2-functor $\Txi:\Mat{\V}\to\Mat{\V}$ as follows. Given a $\V$-relation $r:X\times Y\to\V$, we define $\Txi r:TX\times TY\to\V$ as the left Kan-extension
\[
\xymatrix{T(X\times Y)\ar[rr]^\can\ar[dr]_{\xi_{X\times Y}(r)=\xi\cdot Tr} & & 
TX\times TY\ar@{..>}[dl]^{\Txi r}\\ & \V}
\]
in $\ORD$ (where $TX$, $TY$, $T(X\times Y)$ are discrete), i.e.\ the smallest (order-preserving) map $s:TX\times TY\to\V$ such that $\xi\cdot Tr\le g\cdot\can$. Elementwise, we have
\[
\Txi r(\frx,\fry)=\bigvee\left\{\xi\cdot Tr(\frw)\;\Bigl\lvert\;\frw\in T(X\times Y), T\pi_1(\frw)=\frx,T\pi_2(\frw)=\fry\right\}
\]
for each $\frx\in TX$ and $\fry\in TY$. We have the following properties.
\begin{proposition}[\cite{Hof_TopTh}]
The following assertions hold.
\begin{enumerate}
\item For each $\V$-matrix $r:X\relto Y$, $\Txi(r^\circ)=\Txi(r)^\circ$ (and we write $\Txi r^\circ$).
\item For each function $f:X\to Y$, $Tf= \Txi f$ (and therefore also $Tf^\circ= \Txi f^\circ$).
\item $e_Y\cdot r\leq \Txi r\cdot e_X$ for all $r:X\relto Y$ in $\Mat{\V}$.
\item $m_Y\cdot \Txi^2 r=\Txi r\cdot m_X$ for all $r:X\relto Y$ in $\Mat{\V}$.
\end{enumerate}
\end{proposition}

\subsection{$\Tth$-relations} We define a \emph{$\Tth$-relation} from $X$ to $Y$ to be a $\V$-relation of the form $a:TX\relto Y$, and write $a:X\krelto Y$. Given also $b:Y\krelto Z$, the composite $b\kleisli a:X\krelto Z$ is given by the Kleisli convolution
\[
 b\kleisli a = b\cdot\Txi a \cdot m_X^\circ.
\]
Composition of $\Tth$-relations is associative, and for each $\Tth$-matrix $a:X\krelto Y$ we have $a\kleisli e_X^\circ=a$ and $e_Y^\circ\kleisli a\ge a$, hence $e_X^\circ:X\krelto X$ is a lax identity. We call a $\Tth$-relation $a:X\krelto Y$ \emph{unitary} if $e_Y^\circ\kleisli a= a$, so that $e_X^\circ:X\krelto X$ is the identity on $X$ in the category $\UMat{\Tth}$ of sets and unitary $\Tth$-relations, with the Kleisli convolution as composition. The hom-sets of $\UMat{\Tth}$ inherit the order-structure from $\Mat{\V}$ , and composition of (unitary) $\Tth$-relations respects this order in both variables. Many notions and arguments can be transported from the $\V$-setting to the $\Tth$-setting by substituting relational composition by Kleisli convolution.

Given a $\Tth$-relation $c:X\krelto Z$, the composition by $c$ from the right side has a right adjoint but composition by $c$ from the left side in general not. Explicitely, given also $b:X\krelto Y$, we pass from
\begin{align*}
\xymatrix{X\ar@{-^{>}}|-{\object@{|}}[r]^b\ar@{-^{>}}|-{\object@{|}}[d]_c & Y\\ Z} &&\text{to}&&
\xymatrix{TX\ar[r]|-{\object@{|}}^b\ar[d]|-{\object@{|}}_{m_X^\circ} & Y\\
TTX\ar[d]|-{\object@{|}}_{\Txi c}\\ TZ}\\
\text{in $\UMat{\Tth}$} &&&& \text{in $\Mat{\V}$}
\end{align*}
and define the extension $b \homkleislileft c:Z\krelto Y$ as $b\homcompleft(\Txi c\cdot m_X^\circ):TZ\relto Y$.

\subsection{$\Tth$-categories}\label{TCat}
A $\Tth$-category $X=(X,a)$ is a set $X$ equipped with a $\Tth$-relation $a:X\krelto X$ satisfying $e_X^\circ\le a$ and $a\kleisli a\le a$, equivalentely,
\begin{align*}
 k\le a(e_X(x),x), && \Txi a(\frX,\frx)\otimes a(\frx,x)\le a(\frx,x)
\end{align*}
for all $\frX\in TT X$, $\frx\in TX$ and $x\in X$. A $\Tth$-functor $f:(X,a)\to(Y,b)$ must satisfy $f\cdot a\le b\cdot Tf$, which in pointwise notation reads as
\[
 a(\frx,x)\le b(Tf(\frx),f(x))
\]
for all $\frx\in TX$ and $x\in X$. The resulting category of $\Tth$-categories and $\Tth$-functors is denoted by $\Cat{\Tth}$ (see also \cite{CH_TopFeat,CT_MultiCat,CHT_OneSetting}). Note that the quantale $\V$ becomes in a natural way a $\Tth$-category $\V=(\V,\hom_\xi)$ where $\hom_\xi:T\V\times\V\to\V,\;(\frv,v)\mapsto(\xi(\frv)\multimap v)$.
\begin{examples}
\begin{enumerate}
\item For each quantale $\V$, $\Ith_\V$-categories are precisely $\V$-categories and $\Ith_\V$-functors are $\V$-functors. 
\item The main result of \cite{Bar_RelAlg} states that $\Cat{\Uth_\two}$ is isomorphic to the category $\TOP$ of topological spaces. The $\Uth_\two$-category $\V=\two$ is the Sierpinski space with $\{0\}$ open and $\{1\}$ closed. In \cite{CH_TopFeat} it is shown that $\Cat{\Uth_{\Pplus}\!}$ is isomorphic to the category $\AP$ of approach spaces (see \cite{Low_ApBook} for more details about $\AP$).
\end{enumerate}
\end{examples}
A $\Tth$-category $X=(X,a)$ can be also thought of as a lax Eilenberg--Moore algebra, since the two conditions above can be equivalentely expressed as
\begin{align*}
 1_X\le a\cdot e_X, && a\cdot\Txi a\le a\cdot m_X.
\end{align*}
As a consequence, each $\mT$-algebra $(X,\alpha)$ can be considered as a $\Tth$-category by simply regarding the function $\alpha:TX\to X$ as a $\Tth$-relation $\alpha:X\krelto X$. The free Eilenberg-Moore algebra $(TX,m_X)$ --  viewed as a $\Tth$-category -- is denoted by $|X|$.

Each $\Tth$-category $X=(X,a)$ has an underlying $\V$-category $\ForgetToV\!X=(X,a\cdot e_X)$. Indeed, this defines a functor $\ForgetToV:\Cat{\Tth}\to\Cat{\V}$ which has a left adjoint $\ForgetToVAd:\Cat{\V}\to\Cat{\Tth}$ defined by $\ForgetToVAd\!X=(X,e_X^\circ\cdot\Txi r)$, for each $\V$-category $X=(X,r)$. There is yet another interesting functor connecting $\Tth$-categories and $\V$-categories, namely $\MFunctor:\Cat{\Tth}\to\Cat{\V}$ which sends a $\Tth$-category $(X,a)$ to the $\V$-category $(TX,\Txi a\cdot m_X^\circ)$. The \emph{dual $\Tth$-category} $X^\op$ (see \cite{CH_Compl}) of a $\Tth$-category $X=(X,a)$ is then defined as \[X^\op=\ForgetToVAd(\MFunctor(X)^\op).\]
\begin{examples}\label{Zariski_down}
For $\mT=\mU$ the ultrafilter monad, the topology on $|X|$ can be described via the Zariski-closure:
\[
\frx\in\cl(\calA)\iff \frx\supseteq\bigcap\calA \iff \bigcup\calA\subseteq\frx,
\]
for $\frx\in UX$ and $\calA\subseteq UX$. Furthermore, for $X\in\Cat{\Uth_\two}\cong\TOP$, $M(X)=(UX,\le)$ is the (pre)ordered set where
\[
\frx\le\fry\iff \forall A\in\frx\,.\,\overline{A}\in\fry
\]
for $\frx,\fry\in UX$. Then $X^\op$ is the Alexandroff space induced by the dual order $\ge$. If $X\in\Cat{\Uth_{\Pplus}}\cong\AP$ is an approach space with distance function $\dist:PX\times X\to\Pplus$, then $M(X)=(UX,d)$ is the (generalized) metric space with
\[
d(\frx,\fry)=\inf\{\eps\in [0,\infty]\mid \forall A\in\frx\,.\,\overline{A}^{(\eps)}\in\fry\},
\]
where $\frx,\fry\in UX$ and $\overline{A}^{(\eps)}=\{x\in X\mid \dist(A,x)\le\eps\}$.
\end{examples}

The tensor product of $\V$ can be transported to $\Cat{\Tth}$ by putting $(X,a)\otimes(Y,b)=(X\times Y,c)$ with
\begin{equation}\label{TensAlg}
c(\frw,(x,y))=a(\frx,x)\otimes b(\fry,y),
\end{equation}
where $\frw\in T(X\times Y)$, $x\in X$, $y\in Y$, $\frx=T\pi_1(\frw)$ and $\fry=T\pi_2(\frw)$. The $\Tth$-category $E=(E,k)$ is a $\otimes$-neutral object, where $E$ is a singleton set and  $k$ the constant relation with value $k\in\V$. Unlike the $\V$-case, this does not result in general in a closed structure on $\Cat{\Tth}$. However, as shown in \cite{Hof_TopTh}, if a $\Tth$-category $X=(X,a)$ satisfies $a\cdot\Txi a=a\cdot m_X$, then $X\otimes\_:\Cat{\Tth}\to\Cat{\Tth}$ has a right adjoint $\_^X:\Cat{\Tth}\to\Cat{\Tth}$. Explicitly, for a $\Tth$-category $Y=(Y,b)$, the exponential $X\multimap Y$ is given by the set
\[
\{f:X\to Y\mid f\text{ is a $\Tth$-functor}\}
\]
equipped with the structure-relation $\fspstr{a}{b}$ defined as
\[
\fspstr{a}{b}(\frp,h)=\bigvee\left\{v\in\V\;\Bigl\lvert\;\forall\frq\in T\pi_2^{-1}(\frp),x\in X\;.\; a(T\pi_1(\frq),x)\otimes v\le b(T\!\ev(\frq),h(x))\right\},
\]
where $\frp\in T(Y^X)$, $h\in Y^X$, $\pi_1:X\times(X\multimap Y)\to X$ and $\pi_2:X\times(X\multimap Y)\to Y^X$. Using the adjunction $u\otimes\_\dashv u\multimap\_$ in $\V$, we see that 
\[
\fspstr{a}{b}(\frp,h)=\bigwedge_{\substack{\frq\in T(X\times(X\multimap Y)),x\in X\\ \frq\mapsto \frp}}\hspace{-4ex}
a(T\pi_1(\frq),x)\multimap b(T\!\ev(\frq),h(x)).
\]
\begin{lemma}\label{FunSp_PrinEl}
Let $X=(X,a)$, $Y=(Y,b)$ be $\Tth$-categories with $a\cdot\Txi a=a\cdot m_X$ and $h,h'\in(X\multimap Y)$. Then \[\fspstr{a}{b}(e_{Y^X}(h),h')=\bigwedge_{x\in X}b(e_Y(h(x)),h'(x)).\]
\end{lemma}

\subsection{$\Tth$-modules} Let $X=(X,a)$ and $Y=(Y,b)$ be $\Tth$-categories and $\varphi:X\krelto Y$ be a $\Tth$-relation. We call $\varphi$ a \emph{$\Tth$-module}, and write $\varphi:X\kmodto Y$, if $\varphi\kleisli a\le\varphi$ and $b\kleisli \varphi\le \varphi$. Note that we have always $\varphi\kleisli a\ge\varphi$ and $b\kleisli \varphi\ge \varphi$, so that the $\Tth$-module condition above implies equality. It is easy to see that the extension as well as the lifting (if it exists) in $\UMat{\Tth}$ of $\Tth$-modules is again a $\Tth$-module. Furthermore, we have $a:X\kmodto X$ for each $\Tth$-category $X=(X,a)$; in fact, $a$ is the identity $\Tth$-module on $X$ for the Kleisli convolution. The category of $\Tth$-categories and $\Tth$-modules, with Kleisli convolution as composition is denoted by $\Mod{\Tth}$. In fact, $\Mod{\Tth}$ is an ordered category, with the structure on $\hom$-sets inherited from $\UMat{\Tth}$. 

Let now $X=(X,a)$ and $Y=(Y,b)$ be $\Tth$-categories and $f:X\to Y$ be a $\SET$-map. We define $\Tth$-relations $f_*:X\krelto Y$ and $f^*:Y\krelto X$ by putting $f_*=b\cdot Tf$ and $f^*=f^\circ\cdot b$ respectively. Hence, for $\frx\in TX$, $\fry\in TY$, $x\in X$ and $y\in Y$, $f_*(\frx,y)=b(Tf(\frx),y)$ and $f^*(\fry,x)=b(\fry,f(x))$. Given now $\Tth$-modules $\varphi$ and $\psi$, we have 
\begin{align*}
\varphi\kleisli f_*&=\varphi\cdot Tf &&\text{and}&& f^*\kleisli\psi=f^\circ\cdot\psi.
\end{align*}
The latter equality follows from
\[
f^*\kleisli\psi=f^\circ\cdot b\cdot \Txi\psi\cdot m_Z^\circ=f^\circ\cdot\psi,
\]
whereby the first equality follows from
\begin{align*}
\varphi\kleisli f_*=\varphi\kleisli(b\cdot Tf)=\varphi\cdot \Txi b\cdot T^2f\cdot m_X^\circ
=\varphi\cdot \Txi b\cdot m_Y^\circ\cdot Tf=\varphi\cdot Tf.
\end{align*}
In particular we have $b\kleisli f_*=f_*$ and $f^*\kleisli b=f^*$, as well as $f_*\kleisli f^*=b\cdot Tf\cdot Tf^\circ\cdot\Txi b\cdot m_Y^\circ\le b$. In the latter case we have even equality provided that $f$ is surjective. As before, one easily verifies
\begin{proposition}
The following assertions are equivalent.
\begin{eqcond}
\item $f:X\to Y$ is a $\Tth$-functor.
\item $f_*$ is a $\Tth$-module $f_*:X\kmodto Y$.
\item $f^*$ is a $\Tth$-module $f^*:Y\kmodto X$.
\end{eqcond}
\end{proposition}
As in the $\V$-case, we have functors
\[
\Cat{\Tth}\xrightarrow{(-)_*}\Mod{\Tth}\xleftarrow{(-)^*}\Cat{\Tth}^\op.
\]
We can transport the order-structure on hom-sets from $\Mod{\Tth}$ to $\Cat{\Tth}$ via the functor $(\_)^*:\Cat{\Tth}^\op\to\Mod{\Tth}$, that is, we define $f\le g$ if $f^*\le g^*$, or equivalentely, if $g_*\le f_*$. With this definition we turn $\Cat{\Tth}$ into an \emph{ordered category}. As usual, we call $\Tth$-functors $f,g:X\to Y$ \emph{equivalent}, and write $f\cong g$, if $f\le g$ and $g\le f$. Hence, $f\cong g$ if and only if $f^*=g^*$, which in turn is equivalent to $f_*=g_*$.
\begin{lemma}
Let $f,g:X\to Y$ be $\Tth$-functors between $\Tth$-categories $X=(X,a)$ and $Y=(Y,b)$. Then
\[
f\le g\iff \forall x\in X\;.\;k\le b(e_Y(f(x)),g(x)).
\]
\end{lemma}
\begin{proof}
If $g_*\le f_*$, then
\[
k\le g_*(e_X(x),g(x))\le f_*(e_X(x),g(x))=b(e_Y(f(x)),g(x)).
\]
On the other hand, if $k\le b(e_Y(g(x)),f(x))$ for each $x\in X$, then
\[
f^*(\fry,x)=b(\fry,f(x))
\le \Txi b(e_{TY}(\fry),e_Y(f(x)))\otimes b(e_Y(f(x)),g(x))
\le b(\fry,g(x))=g^*(\fry,x).\qedhere
\]
\end{proof}
In particular, for $\Tth$-functors $f,g:X\to\V$, we have $f\le g$ if and only if $f(x)\le g(x)$ for all $x\in X$. Assume now that $X=(X,a)$, $Y=(Y,b)$ and $Z=(Z,c)$ are $\Tth$-categories where $a\cdot\Txi a=a\cdot m_X$. By combining the previous lemma with Lemma \ref{FunSp_PrinEl}, we obtain $f\le g\iff \mate{f}\le\mate{g}$ for all $\Tth$-functors $f,g:X\otimes Y\to Z$, where $\mate{f},\mate{g}:Y\to Z^X$.

\subsection{Yoneda}
Also $\Tth$-modules give rise to $\Tth$-functors, but besides $X^\op$ we must take also the $\Tth$-category $|X|$ (see \ref{TCat}) into consideration. 
\begin{theorem}[\cite{CH_Compl}]\label{CharTMod}
For $\Tth$-categories $(X,a)$ and $(Y,b)$, and a $\Tth$-relation $\psi:X\krelto Y$, the following assertions are equivalent.
\begin{eqcond}
\item $\psi:(X,a)\kmodto(Y,b)$ is a $\Tth$-module.
\item Both $\psi:|X|\otimes Y\to\V$ and $\psi:X^\op\otimes Y\to\V$ are $\Tth$-functors.
\end{eqcond}
\end{theorem}
Since we have $a:X\kmodto X$ for each $\Tth$-category $X=(X,a)$, the theorem above provides us with two $\Tth$-functors
\begin{align*}
a:|X|\otimes X\to\V &&\text{and}&& a:X^\op\otimes X\to\V.
\end{align*}
To the mate $\yoneda=\mate{a}:X\to(|X|\multimap \V)$ of the first $\Tth$-functor we refer as the Yoneda functor. We have the following
\begin{theorem}[\cite{CH_Compl}]\label{Yoneda}
Let $X=(X,a)$ be a $\Tth$-category. Then the following assertions hold.
\begin{enumerate}
\item For all $\frx\in TX$ and $\varphi\in(|X|\multimap \V)$,  $\fspstr{m_X}{\hom_\xi}(T\!\yoneda(\frx),\varphi)\le\varphi(\frx)$.
\item Let $\varphi\in(|X|\multimap \V)$. Then
\begin{align*}
\forall\frx\in TX\,.\,\varphi(\frx)\le\fspstr{m_X}{\hom_\xi}(T\!\yoneda(\frx),\varphi) &&\iff&&
\varphi:X^\op\to\V\text{ is a $\Tth$-functor}.
\end{align*}
\end{enumerate}
\end{theorem}
Consequentely, we put $\hat{X}=(\hat{X},\hat{a})$ where \[\hat{X}=\{\psi\in(|X|\multimap\V)\mid \psi:X^\op\to\V\text{ is a $\Tth$-functor}\}\] considered as a subcategory of $|X|\multimap \V$. In particular, $\yoneda:X\to\hat{X}$ is full and faithful.
\begin{example}\label{Psh_vs_Filter}
For $X\in\Cat{\Uth_\two}\cong\TOP$, $\psi\in\hat{X}$ is the characteristic function of a Zariski closed and down-closed subset $\calA\subseteq UX$ (see Examples \ref{Zariski_down}).  We will give now an alternative description of $\hat{X}$, as the set $F_0(X)$ of (possibly improper) filters on the lattice $\tau$ of open sets of X, in terms of the bijective maps
\begin{align*}
\hat{X}\xrightarrow{\;\Phi\;} F_0(X) &&\text{and}&& F_0(X)\xrightarrow{\;\Pi\;}\hat{X},
\end{align*}
where $\Phi(\calA)=\bigcap\calA\cap\tau$ and $\Pi(\frf)=\{\frx\in UX\mid \frf\subseteq\frx\}$. Clearly, $\calA=\Pi(\frf)$ is Zariski closed. If $\frx\le\fry$ for some $\frx\in UX$ and $\fry\in\calA$, then, for each $A\in\frx$ and $B\in\frf$, we have
\[
\overline{A}\cap B\neq\emptyset
\]
which, since $B$ is open, gives $A\cap B\neq \emptyset$. Hence $\frf\subseteq\frx$, that is, $\frx\in\calA$. Furthermore, one easily sees that $\frf=\Phi\Pi(\frf)$ and $\calA\subseteq\Pi\Phi(\calA)$. On the other hand, for $\frx\supseteq\bigcap\calA\cap\tau$ and $A\in\frx$ we have $X\setminus\overline{A}\notin\bigcap\calA$, and therefore $X\setminus\overline{A}\notin\frx$ for some $\frx\in\calA$, hence $\overline{A}\in\frx$. Consequently, $\overline{A}\subseteq\bigcup\calA$ and, since $\calA$ is Zariski closed, $\frx\le\fry$ for some $\fry\in\calA$. But $\calA$ is also down-closed, hence $\frx\in\calA$.
In a similar way (in fact, even easier) one can show that there are bijective maps
\begin{align*}
\check{X}\xrightarrow{\;\Phi'\;} F_1(X) &&\text{and}&& F_1(X)\xrightarrow{\;\Pi'\;}\check{X},
\end{align*}
where $\check{X}=\{\calA\subseteq UX\mid \calA\text{ is Zariski closed and up-closed}\}$, $F_1(X)$ is the set of all (possibly improper) filters on the lattice $\sigma$ of closed sets of X, $\Phi'(\calA)=\bigcap\calA\cap\sigma$ and $\Pi'(\frf)=\{\frx\in UX\mid \frf\subseteq\frx\}$. Furthermore, for any $\calA\subseteq UX$ Zariski-closed, its down-closure $\downc\calA$ is Zariski-closed as well. To see this, let $\frx\in\cl(\downc\calA)$. Hence $\frx\in\bigcup\downc\calA$ and therefore, for any $A\in\frx$, we have $\overline{A}\in\bigcup\calA$. Define
\[
\frj=\{B\subseteq X\mid\forall\fra\in\calA\,.\,B\notin\fra\}.
\]
Then $\frj$ is an ideal, and $\frj\cap\{\overline{A}\mid A\in\frx\}=\varnothing$. Hence there is some $\fry\in UX$ such that $\frx\le\fry$ and $\frj\cap\fry=\varnothing$. But the latter fact gives us $\fry\subseteq\bigcup\calA$, that is, $\fry\in\cl\calA=\calA$. We conclude $\frx\in\downc\calA$. With a similar proof one can show that $\upc\calA$ is Zariski closed for each Zariski closed subset $\calA\subseteq UX$ (but now use $\frx\in\cl(\upc\calA)\iff\bigcap\upc\calA\subseteq\frx$).\\
The topology in $\hat{X}$ is the \emph{compact-open topology}, which has as basic open sets
\begin{align*}
B(\calB,\{0\})=\{\calA\in\hat{X}\mid \calA\cap\calB=\varnothing\},&&\text{$\calB\subseteq UX$ Zariski closed.}
\end{align*}
Since $B(\calB,\{0\})=B(\upc\calB,\{0\})$, it is enough to consider Zariski closed and up-closed subsets $\calB\subseteq UX$. Hence, using the bijections $\hat{X}\cong F_0(X)$ and $\check{X}\cong F_1(X)$, $F_0(X)$ has
\begin{align*}
\{\frf\in F_0(X)\mid \exists A\in\frf,B\in\frg\,.\,A\cap B=\varnothing\}&&(\frg\in F_1(X))
\end{align*}
as basic open sets. Clearly, it is enough to consider $\frg=\doo{B}$ the principal filter induced by a closed set $B$, so that all sets
\begin{align*}
\{\frf\in F_0(X)\mid \exists A\in\frf\,.\,A\cap B=\varnothing\}=\{\frf\in F_0(X)\mid X\setminus B\in\frf\}&& \text{($B\subseteq X$ closed)}
\end{align*}
form a basis for the topology on $F_0(X)$. But this is precisely the topology on $F_0(X)$ considered in \cite{Esc_InjSp}.
\end{example}

\subsection{L-separation}
We call a $\Tth$-category $X=(X,a)$ \emph{L-separated} whenever the ordered set $\Cat{\Tth}(Y,X)$ is anti-symmetric, for each $\Tth$-category $Y$. The full subcategory of $\Cat{\Tth}$ consisiting of all L-separated $\Tth$-categories is denotd by $\CatSep{\Tth}$.
\begin{proposition}
Let $X=(X,a)$ be a $\Tth$-category. Then the following assertions are equivalent.
\begin{eqcond}
\item $X$ is L-separated.
\item $x\cong y$ implies $x=y$, for all $x,y\in X$.
\item For all $x,y\in X$, if $a(e_X(x),y)\ge k$ and $a(e_X(y),x)\ge k$, then $x=y$.
\item $\yoneda:X\to\hat{X}$ is injective.
\end{eqcond}
\end{proposition}
\begin{proof}
As for Proposition \ref{EqCondSepV}.
\end{proof}
\begin{corollary}
\begin{enumerate}
\item The $\Tth$-category $\V=(\V,\hom_\xi)$ is separated.
\item For all $\Tth$-categories $Y=(Y,b)$ and $X=(X,a)$ where $Y$ is L-separated and $a\cdot \Txi a=a\cdot m_X$, $Y^X$ is L-separated. In particular, $|X|\multimap\V$ is L-separated, for each $\Tth$-category $X$.
\item Any subcategory of a L-separated $\Tth$-category is L-separated. In particular, $\hat{X}$ is L-separated, for each $\Tth$-category $X$.
\end{enumerate}
\end{corollary}
\begin{examples}
A topological space is L-separated if and only if it is T$_0$, whereas an approach space $X=(X,d)$ with distance function $d:PX\times X\to\Pplus$ is L-separated if and only if
\[
d(\{x\},y)=0=d(\{y\},x)\;\Rw\;x=y
\]
for all $x,y\in X$.
\end{examples}

\subsection{L-completeness}
As in \ref{ComplV}, we call a $\Tth$-category $X=(X,a)$ \emph{L-complete} if every adjunction $\varphi\dashv\psi$ with $\varphi:Z\kmodto X$ and $\psi:X\kmodto Z$ is of the form $f_*\dashv f^*$ for a $\Tth$-functor $f:Z\to X$. Of course, $f$ is up to equivalence uniquely determined by $\varphi\dashv\psi$, and is indeed unique if $X$ is L-separated. As before, it is enough to consider $Z=E$ (see also \cite{CH_Compl})
\begin{proposition}
Let $X=(X,a)$ be a $\Tth$-category. The following assertions are equivalent.
\begin{eqcond}
\item $X$ is L-complete.
\item Each left adjoint $\Tth$-module $\varphi:E\kmodto X$ is of the form $\varphi=x_*$ for some $x$ in $X$.
\item Each right adjoint $\Tth$-module $\psi:X\kmodto E$ is of the form $\psi=x^*$ for some $x$ in $X$.
\end{eqcond}
\end{proposition}
A topological space is L-complete precisely if it is weakly-sober, that is, if each irreducible closed set is the closure of a point. A similar result holds for approach spaces: L-completeness is equivalent to each irreducible closed variable set $A$ is representable (see \cite{CH_Compl} for details). Furthermore, in both cases we have that L-complete and L-separated (approach respectively topological) spaces are precisely the fixed objects of the dual adjunction with the category of (approach) frames induced by $\V=\two$ respectively $\V=\Pplus$ (see \cite{Olm_AppFrame} for details about the approach case).

For a pair $\psi:X\kmodto Y$ and $\varphi:Y\kmodto X$ of adjoint $\Tth$-modules $\varphi\dashv\psi$, the same calculation as in \ref{ComplV} shows that $\varphi=1_X^*\homkleislileft\psi$. Since for each $\Tth$-module $\psi:X\kmodto Y$ we have $(1_X^*\homkleislileft\psi)\kleisli\psi\le1_X^*$, $\psi$ is right adjoint if and only if $\psi\kleisli(1_X^*\homkleislileft\psi)\ge(1_Y)_*$. Considering in particular $Y=E$, a $\Tth$-module $\psi:X\kmodto E$ is right adjoint if and only if
\[
k\le\bigvee_{\frx\in TX}\psi(\frx)\otimes\xi\cdot T\varphi(\frx)
\]
where $\varphi=1_X^*\homkleislileft\psi$. Note that $\bigvee\{\xi\cdot T\psi(\frX)\mid \frX\in TTX,m_X(\frX)=\frx\}=\psi(\frx)$ since $\psi:|X|\to\V$ is a $\Tth$-functor, hence, with the help of Lemma \ref{FunSp_PrinEl}, we see that
\begin{align*}
\varphi(x)&=\bigwedge_{\frx\in TX}\left(\left(\bigvee_{\substack{\frX\in TTX,\\m_X(\frX)=\frx}}\xi\cdot T\psi(\frX)\right)\multimap a(\frx,x)\right)\\
&=\bigwedge_{\frx\in TX}(\psi(\frx)\multimap a(\frx,x))\\
&=\hat{a}(e_{\hat{X}}(\psi),\yoneda(x)).
\end{align*}
\begin{lemma}\label{XiPhi}
Let $\psi:X\kmodto E$ be a $\Tth$-module and put $\varphi=1_X^*\homkleislileft\psi$. Then, for each $\frx\in TX$,
\[
\xi\cdot T\varphi(\frx)=\Txi\hat{a}(e_{T\hat{X}}\cdot e_{\hat{X}}(\psi),T\yoneda(\frx)).
\]
\end{lemma}
\begin{proof}
Since $\xi\cdot T\varphi(\frx)=\Txi\varphi(\frx)$, the result follows from applying $\Txi$ to the equality above.
\end{proof}
Hence we have
\begin{proposition}\label{AdjMod}
Let $X=(X,a)$ be $\Tth$-category. A $\Tth$-module $\psi:X\kmodto E$ is right adjoint if and only if 
\begin{equation}\label{tight}
k\le \bigvee_{\frx\in TX}\psi(\frx)\otimes \Txi\hat{a}(e_{T\hat{X}}\cdot e_{\hat{X}}(\psi),T\yoneda(\frx)).
\end{equation}
\end{proposition}
Given a $\Tth$-category $X=(X,a)$, we call a $\Tth$-functor $\psi:|X|\to\V$ \emph{tight} if $\psi:X^\op\to\V$ is a $\Tth$-functor and if, considered as a $\Tth$-module $\psi:X\kmodto E$, it is right adjoint, that is, if it satisfies \eqref{tight}.
\begin{example}
For a topological space $X$ and $\psi\in\hat{X}$, as before we can identify $\psi$ with a Zariski closed and down-closed subset $\calA\subseteq UX$, and then $1_X^*\homkleislileft\psi$ with
\[
A=\{x\in X\mid\forall\fra\in\calA\,.\,\fra\to x\}.
\]
Then $\psi$ is tight if, and only if, there exists some $\fra\in\calA$ with $A\in\fra$. Furthermore, under the bijection $\hat{X}\cong F_0(X)$ (see Example \ref{Psh_vs_Filter}), a tight map $\psi$ corresponds to a filters $\frf\in F_0(X)$ with $(\Lim\frf)\match\frf$, where $\Lim\frf$ denotes the set of all limit points of $\frf$ and $A\match\frg$ if $\forall B\in\frf\,.\,A\cap B\neq\varnothing$. Furthermore, for each $\frf\in F_0(X)$ we have
\[
(\Lim\frf)\match\frf\iff \frf\text{ is completely prime},
\]
that is, if $\bigcup_{i\in I}U_i\in\frf$, then $U_i\in\frf$ for some $i\in I$. In fact, if $(\Lim\frf)\match\frf$ and $\bigcup_{i\in I}U_i\in\frf$ for some family of open subsets of $X$, then $(\Lim\frf)\cap \bigcup_{i\in I}U_i\neq\varnothing$, and therefore, for some $i\in I$, $U_i$ contains a limit point of $\frf$. Hence $U_i\in\frf$. On the other hand, assume that $\frf$ is completely prime. Suppose that $U\in\frf$ does not contain a limit point of $\frf$. Then, for each $x\in U$, there is an open neiborhood $U_x$ of $x$ with $U_x\notin\frf$. But $\bigcup_{x\in X}U_x\in\frf$ and, since $\frf$ is completely prime, $U_x\in\frf$ for some $x\in U$, a contradiction.
\end{example}

\subsection{L-injectivity}
The notions of L-dense $\Tth$-functor, L-equivalence as well as L-injective $\Tth$-category can now be introduced as in \ref{LInjV}. More precise, we call a $\Tth$-functor $f:(X,a)\to(Y,b)$ \emph{L-dense} if $f_*\kleisli f^*=1_X^*$, which amounts to $b\cdot Tf\cdot Tf^\circ\cdot\Txi b\cdot m_Y^\circ= b$. L-dense $\Tth$-functors have the same composition-cancellation properties as $\V$-functors (see \ref{LInjV}). A fully faithful L-dense $\Tth$-functor is an \emph{L-equivalence}, which can be equivalentely expressed by saying that $f_*$ is an isomorphism in $\Mod{\Tth}$. A $\Tth$-category $Z$ is called \emph{pseudo-injective} if, for all $\Tth$-functors $f:X\to Z$ and fully faithful $\Tth$-functors $i:X\to Y$, there exists a $\Tth$-functor $g:Y\to Z$ such that $g\cdot i\cong f$. $Z$ is called \emph{L-injective} if this extension property is only required along L-equivalences $i:X\to Y$. Of course, for a L-separated $\Tth$-category $Z$, $g\cdot i\cong f$ implies $g\cdot i= f$, and then pseudo-injectivity coincides with the usual notion of injectivity. The following two results can be proven as in \ref{LInjV}
\begin{lemma}
The $\Tth$-category $\V$ is injective.
\end{lemma}
\begin{proposition}
For all $\Tth$-categories $Y=(Y,b)$ and $X=(X,a)$ where $Y$ is L-injective (pseudo-injective) and $a\cdot \Txi a=a\cdot m_X$, $Y^X$ is L-injective (pseudo-injective).
\end{proposition}
In particular, we obtain the injectivity of the $\Tth$-category $|X|\multimap\V$. Later on we will see that also $\hat{X}$ and $\tilde{X}$ are L-injective.

\section{L-closure}

\subsection{L-dense $\Tth$-functors}
As in \ref{LdenseV}, L-dense $\Tth$-functors can be characterized as ``epimorphisms up to $\cong$''. However, we will use here a slighly different proof.
\begin{lemma}\label{CharLdenseT}
Let $X=(X,a)$ be a $\Tth$-category, $M\subseteq X$ and $i:M\hrw X$ the embedding of $M$ into $X$. Then $i$ is dense if and only if
\begin{equation}\label{dense}
k\le\bigvee_{\fra\in TM}a(\fra,x)\otimes\Txi a(Te_X\cdot e_X(x),\fra)
\end{equation}
for all $x\in X$.
\end{lemma}
\begin{proof}
Recall that $i$ is dense whenever $i_*\kleisli i^*\ge a$, that is,
\[
a(\frx,x)\le\bigvee_{\fra\in TM}\bigvee_{\substack{\frX\in TTX\\m_X(\frX)=\frx}}a(\fra,x)\otimes\Txi a(\frX,\fra)
\]
for all $\frx\in TX$ and $x\in X$. If $i$ is dense, then \eqref{dense} follows from the inequality above by putting $\frx=e_X(x)$ and using $m_X^\circ\cdot e_X=e_{TX}\cdot e_X$ (see Subsection \ref{Topth}). On the other hand, from \eqref{dense} we obtain
\begin{align*}
a(\frx,x)
&\le \bigvee_{\fra\in TM}a(\fra,x)\otimes\Txi a(Te_X\cdot e_X(x),\fra)\otimes a(\frx,x)\\
&\le \bigvee_{\fra\in TM}a(\fra,x)\otimes \Txi\Txi a(e_{TTX}\cdot e_{TX}(\frx),e_{TX}\cdot e_X(x))\otimes \Txi a(Te_X\cdot e_X(x),\fra)\\
&\le \bigvee_{\fra\in TM}a(\fra,x)\otimes \Txi a(e_{TX}(\frx),\fra)\\
&\le \bigvee_{\fra\in TM}\bigvee_{\substack{\frX\in TTX\\m_X(\frX)=\frx}}a(\fra,x)\otimes\Txi a(\frX,\fra).\qedhere
\end{align*}
\end{proof}

\begin{proposition}
For a $\Tth$-functor $i:M\to X$, the following assertions are equivalent.
\begin{eqcond}
\item $i:M\to X$ is L-dense.
\item For all $\Tth$-functors $f,g:X\to Y$, with $f\cdot i=g\cdot i$ one has $f\cong g$.
\item For all $\Tth$-functors $f,g:X\to \V$, with $f\cdot i=g\cdot i$ one has $f= g$.
\end{eqcond}
\end{proposition}
\begin{proof}
Assume first (i), i.e.\ $i:M\to X$ is L-dense. Then, from $f\cdot i=g\cdot i$ we obtain $f_*=g_*$ since $i_*\kleisli i^*=1_X^*$. The implication (ii)$\Rw$(iii) is trivially true. Assume now (iii). According to the remarks made above, we can assume that $i:M\to X$ is the embedding of a subset $M\subseteq X$. Let $x\in X$. First note that
\[
\varphi:X\to\V,\;y\mapsto a(e_X(x),y)
\]
$\Tth$-functor since $a:|X|\otimes X\to\V$ is so. Using the same argument as in \cite[Lemma 6.8]{Hof_TopTh}, we see that also 
\[
\psi:X \to\V,\;y\mapsto\bigvee_{\frx\in TM}\Txi a(Te_X\cdot e_X(x),\frx)\otimes a(\frx,y)
\]
is a $\Tth$-functor. Clearly, for each $y\in X$ we have $\psi(y)\le\varphi(y)$. If $y\in M$, we can choose $\frx=e_X(y)\in TM$ and therefore, using $Te_X\cdot e_X=e_{TX}\cdot e_X$ and op-laxness of $e$, obtain $\varphi(y)\le\psi(y)$. Hence $\varphi|_M=\psi|_M$, and from our assumption (iii) we deduce $k\le\varphi(x)=\psi(x)$.
\end{proof}

\subsection{L-closure}
For a $\Tth$-category $X=(X,a)$ and $M\subseteq X$, we define the \emph{L-closure} of $M$ in $X$ by
\[
\overline{M}=\{x\in X\mid \forall f,g:X\to Y\,.\,(f|_M=g|_M\;\Rw\;f(x)\cong g(x))\}.
\]
Hence $\overline{M}$ is the largest subset $N$ of $X$ making the inclusion map $i:M\hrw N$ dense.
\begin{proposition}\label{CharLClsT}
Let $X=(X,a)$ be a $\Tth$-category, $M\subseteq X$ and $x\in X$. Then the following assertions are equivalent.
\begin{eqcond}
\item $x\in\overline{M}$.
\item $k\le\bigvee_{\frx\in TM}a(\frx,x)\otimes\Txi a(Te_X\cdot e_X(x),\frx)$
\item $i^*\kleisli x_*\dashv x^*\kleisli i_*$, where $i:M\hrw X$ is the inlcusion map.
\item $1_E^*\le x^*\kleisli i_*\kleisli i^*\kleisli x_*$, 
\item $i^*\kleisli x_*\dashv x^*\kleisli i_*$.
\item $x_*:E\modto X$ factors through $i_*:M\modto X$ by a map $\varphi:E\modto M$ in $\Mod{\Tth}$. 
\end{eqcond}
\end{proposition}
\begin{proof}
As for Proposition \ref{CharLClsT}, using now Lemma \ref{CharLdenseT}.
\end{proof}
We can now proceed as in \ref{LclosureV}.
\begin{proposition}
For a $\Tth$-functor $f : X\to Y$ and $M,M'\subseteq X$, $N\subseteq Y$, we have
\begin{enumerate}
\item $M\subseteq\overline{M}$ and $M\subseteq M'$ implies $\overline{M}\subseteq\overline{M'}$.
\item $\overline{\varnothing}=\varnothing$ and $\overline{\overline{M}}=\overline{M}$.
\item $f(\overline{M})\subseteq \overline{f(M)}$ and $f^{-1}(\overline{N})\supseteq\overline{f^{-1}(N)}$.
\item If $k$ is $\vee$-irreducible and $T$ preserves finite sums, then $\overline{M\cup M'}=\overline{M}\cup\overline{M'}$.
\end{enumerate}
\end{proposition}
\begin{corollary}
If $k$ is $\vee$-irreducible in $\V$, then the L-closure operator defines a topology on $X$ such that every $\Tth$-functor becomes continuous. Hence, L-closure defines a functor $L:\Cat{\Tth}\to\TOP$.
\end{corollary}
\begin{example}
For a topological space $X$, $x\in X$ lies in the L-closure of $A\subseteq X$ precisely if there exists some ultrafilter $\frx\in UA$ with $\overline{x}\in\frx$ and which converges to $x$; in other words, for each neiborhood $U$ of $x$ we have $U\cap\overline{x}\cap A\neq\varnothing$. Hence the L-closure of a topological space $X$ coincides with the so called \emph{b-closure} \cite{Bar_EpiT0}.
\end{example}

\subsection{L-separatedness via the L-closure}
\begin{proposition}
Let $X=(X,a)$ be a $\Tth$-category and $\Delta\subseteq X\times X$ the diagonal. Then
\[
\overline{\Delta}=\{(x,y)\in X\times Y\mid x\cong y\}
\]
\end{proposition}
\begin{proof}
As for Proposition \ref{ClDiagV}.
\end{proof}
\begin{corollary}
A $\Tth$-category $X$ is L-separated if and only if the diagonal $\Delta$ is closed in $X\times X$. 
\end{corollary}

\begin{theorem}
$\CatSep{\Tth}$ is an epi-reflective subcategory of $\Cat{\Tth}$, where the reflection map is given by $\yoneda_X:X\to\yoneda_X(X)$, for each $\Tth$-category $X$. Hence, limits of L-separated $\Tth$-categories are formed in $\Cat{\Tth}$, while colimits are obtained by reflecting the colimit formed in $\Cat{\Tth}$. The epimorphisms in $\Cat{\Tth}$ are precisely the L-dense $\Tth$-functors. 
\end{theorem}

\subsection{L-completeness via the L-closure}

\begin{lemma}
Let $X=(X,a)$ be a $\Tth$-category and $M\subseteq X$. Then the following assertions hold.
\begin{enumerate}
\item Assume that $X$ is L-complete and $M$ be L-closed. Then $M$ is L-complete.
\item Assume that $X$ is L-separated and $M$ is L-complete. Then $M$ is L-closed.
\end{enumerate}
\end{lemma}
\begin{proof}
As for Lemma \ref{ClosVsComplV}.
\end{proof}
\begin{theorem}\label{CharComplT}
Let $X=(X,b)$ be a $\Tth$-category. The following assertions are equivalent.
\begin{eqcond}
\item $X$ is L-complete.
\item $X$ is L-injective.
\item $\yoneda:X\to\tilde{X}$ has a pseudo left-inverse $\Tth$-functor $R:\tilde{X}\to X$, i.e.\ $R\cdot\yoneda\cong 1_X$.
\end{eqcond}
\end{theorem}
\begin{proof}
To see (i)$\Rw$(ii), let $i:A\to B$ be a fully faithful dense $\Tth$-functor and $f:A\to X$ be a $\Tth$-functor. Since $i_*\dashv i^*$ is actually an equivalence of $\Tth$-modules, we have $f_*\kleisli i^*\dashv i_*\kleisli f^*$. Hence, since $X$ is L-complete, there is a $\Tth$-functor $g:B\to X$ such that $g_*=f_*\kleisli i^*$, hence $g_*\kleisli i_*=f_*$.\\
The implication (ii)$\Rw$(iii) is surely true since $\yoneda:X\to\tilde{X}$ is dense and fully faithful.\\
Finally, to see (iii)$\Rw$(i), let $R:\tilde{X}\to X$ be a left inverse of $\yoneda:X\to\tilde{X}$. Then $\yoneda\cdot R=1_{\tilde{X}}$ since $\yoneda:X\to\tilde{X}$ is dense and $\tilde{X}$ is L-separated. Hence, for each right adjoint $\Tth$-module $\psi:X\kmodto E$, we have $\psi=R(\psi)^*$.
\end{proof}
Therefore we have that $|X|\multimap\V$ is L-complete. Our next result shows that also $\hat{X}$ is L-complete.
\begin{proposition}
$\hat{X}$ is L-closed in $|X|\multimap\V$, for each $\Tth$-category $X$.
\end{proposition}
\begin{proof}
Let $X=(X,a)$ be a $\Tth$-category and assume that $\varphi\in(|X|\multimap\V)$ belongs to the closure of $\hat{X}$, that is,
\[
k\le\bigvee_{\fru\in T\hat{X}}\fspstr{m_X}{\hom_\xi}(\fru,\varphi)\otimes \Txi\fspstr{m_X}{\hom_\xi}(Te_{|X|\multimap\V}\cdot e_{|X|\multimap\V}(\varphi),\fru).
\]
We wish to show that $r(\frx,\fry)\otimes\varphi(\fry)\le\varphi(\frx)$ for all $\frx,\fry\in TX$, where $r=\Txi a\cdot m_X^\circ$.\\
First note that, for all $\alpha,\beta\in(|X|\multimap\V)$, 
\[
e_{|X|\multimap\V}^\circ\cdot\fspstr{m_X}{\hom_\xi}(\alpha,\beta)
=\bigwedge_{\frx\in TX}(\alpha(\frx)\multimap\beta(\frx)).
\]
Hence, with $h_\frx:(|X|\multimap\V)\relto(|X|\multimap\V)$, $h_\frx(\alpha,\beta)=(\alpha(\frx)\multimap\beta(\frx))$, we have $T e_{|X|\multimap\V}^\circ\cdot\Txi\fspstr{m_X}{\hom_\xi}\le\Txi h_\frx$. Since the diagram
\[
\xymatrix{(|X|\multimap\V)\times(|X|\multimap\V)\ar[rr]^-{\ev_\frx\times\ev_\frx}\ar[drr]_{h_\frx} &&
\V\times\V\ar[d]^\multimap \\ &&\V}
\]
commutes and in
\[
\xymatrix{T((|X|\multimap\V)\times(|X|\multimap\V))\ar[rr]^-{T(\ev_\frx\times\ev_\frx)}\ar[d]_\can
&& T(\V\times\V)\ar[r]^-{T(\multimap)}\ar[d]_{\can}\ar@{}[ddr]|{\ge} & T\V\ar[dd]^\xi\\
T(|X|\multimap\V)\times T(|X|\multimap\V)\ar[rr]_-{T\ev_\frx\times T\ev_\frx} && T\V\times T\V\ar[d]_{\xi\times\xi}\\
&&\V\times\V\ar[r]_\multimap &\V}
\]
the left hand side diagram commutes and in the right hand side diagram we have ``lower path'' greater or equal ``upper path'', we have
\[
\Txi h_\frx(\fru,\frv)\le(\fru(\frx)\multimap\frv(\frx))
\]
for each $\frx\in TX$ and $\fru,\frv\in T(|X|\multimap\V)$, where $\fru(\frx)=\xi\cdot T\ev_\frx(\fru)$. Accordingly, $e_{|X|\multimap\V}(\varphi)(\frx)=\varphi(\frx)$ and we obtain
\begin{align*}
\forall\frx\in TX\;.\;\Txi\fspstr{m_X}{\hom_\xi}(Te_{|X|\multimap\V}\cdot e_{|X|\multimap\V}(\varphi),\fru)\le(\varphi(\frx)\multimap\fru(\frx)).
\end{align*}
Furthermore, for all $\frx,\fry\in TX$ we have
\[
\xymatrix{\hat{X}\ar[d]_{!}\ar[r]^-{\Delta}\ar@{}[drr]|{\le} &\hat{X}\times\hat{X}\ar[r]^-{\ev_\frx\times\ev_\fry} &\V\times\V\ar[d]^\multimap\\
1\ar[rr]_{r(\frx,\fry)} && \V}
\]
and we obtain
\[
r(\frx,\fry)
\le \xi\cdot T(\multimap)\cdot T(\ev_\frx\times\ev_\fry)\cdot T\Delta(\fru)
\le (\fru(\fry)\multimap\fru(\frx))
\]
for each $\fru\in T\hat{X}$. We conclude that
\begin{align*}
r(\frx,\fry)\otimes\varphi(\fry)
&\le \bigvee_{\fru\in T\hat{X}}r(\frx,\fry)\otimes\varphi(\fry)\otimes\fspstr{m_X}{\hom_\xi}(\fru,\varphi)\otimes \Txi\fspstr{m_X}{\hom_\xi}(Te_{|X|\multimap\V}\cdot e_{|X|\multimap\V}(\varphi),\fru)\\
&\le \bigvee_{\fru\in T\hat{X}}r(\frx,\fry)\otimes\varphi(\fry)\otimes (\varphi(\fry)\multimap\fru(\fry))\otimes \fspstr{m_X}{\hom_\xi}(\fru,\varphi)\\
&\le \bigvee_{\fru\in T\hat{X}}r(\frx,\fry)\otimes\fru(\fry)\otimes \fspstr{m_X}{\hom_\xi}(\fru,\varphi)\\
&\le \bigvee_{\fru\in T\hat{X}}\fru(\frx)\otimes \fspstr{m_X}{\hom_\xi}(\fru,\varphi)
\le \bigvee_{\fru\in T\hat{X}}\fru(\frx)\otimes(\fru(\frx)\multimap\varphi(\frx))\le\varphi(\frx).\qedhere
\end{align*}
\end{proof}
\begin{proposition}
Let $X=(X,a)$ be a $\Tth$-category and $\psi\in\hat{X}$. Then $\psi$ is a right adjoint $\Tth$-module if and only if $\psi\in\overline{\yoneda(X)}$.
\end{proposition}
\begin{proof}
By Proposition \ref{AdjMod} and Theorem \ref{Yoneda}, $\psi$ is right adjoint if and only if
\begin{align*}
k&\le\bigvee_{\frx\in TX}\hat{a}(T\yoneda(\frx),\psi)\otimes \Txi\hat{a}(Te_{\hat{X}}\cdot e_{\hat{X}}(\psi),T\yoneda(\frx)),
\end{align*}
which means precisely that $\psi\in\overline{\yoneda(X)}$.
\end{proof}
The proposition above identifies $\tilde{X}$ as the L-closure of $\yoneda(X)$ in $\hat{X}$, and therefore as an L-complete $\Tth$-category. Furthermore, $\yoneda:X\to\tilde{X}$ is fully faithful and L-dense. Hence we have
\begin{theorem}
The full subcategory $\CatCompl{\Tth}$ of $\CatSep{\Tth}$ of L-complete $\Tth$-categories is an epi-reflective subcategory of $\CatSep{\Tth}$. The reflection map of a L-separated $\Tth$-category $X$ is given by any full L-dense embedding of $X$ into an L-complete and L-separated $\Tth$-category, for instance by $\yoneda:X\to\tilde{X}$.
\end{theorem}


\end{document}